\documentclass[11pt, reqno, a4paper]{amsart}

\usepackage{amssymb, amsmath, amscd}
\usepackage{amsthm}

\usepackage{verbatim}

\usepackage{color}

\usepackage{etex}
\usepackage{dsfont}
\usepackage[matrix, arrow, curve, frame, arc]{xy}
\usepackage{pgf}
\usepackage{tikz-cd}
\usetikzlibrary{arrows,shapes,automata,backgrounds,petri,calc}
\usepackage[square,sort,comma,numbers]{natbib}
 \usepackage{url}
 \usepackage{hyperref}
 \usepackage[shortlabels]{enumitem} %to index the enumeration

\newcommand{\tikzAngleOfLine}{\tikz@AngleOfLine}
\def\tikz@AngleOfLine(#1)(#2)#3{%
\pgfmathanglebetweenpoints{%
\pgfpointanchor{#1}{center}}{%
\pgfpointanchor{#2}{center}}
\pgfmathsetmacro{#3}{\pgfmathresult}%
}

\newcommand{\bN}{\mathbb{N}}

\newcommand{\cB}{\mathcal{B}}
\newcommand{\wt}{\widetilde}
\newcommand{\ba}{\bar{\alpha}}
\newcommand{\La}{\Lambda}
\newcommand{\vf}{\varphi}
\newcommand{\ve}{\varepsilon}
\newcommand{\cT}{\mathcal{T}}

\newcommand{\cR}{\mathcal{R}}
\newcommand{\cP}{\mathcal{P}}

\newcommand{\bD}{\mathbb{D}}
\newcommand{\bE}{\mathbb{E}}

\newcommand{\rad}{\operatorname{rad}}
\newcommand{\soc}{\operatorname{soc}}
\renewcommand{\mod}{\operatorname{mod}}
 
\newcommand{\Img}{\operatorname{Im}}

\def\vec#1{\left[\begin{smallmatrix}#1\end{smallmatrix}\right]}

\begin{document}

\newtheorem{defi}{Definition}[section]
\newtheorem{rem}[defi]{Remark}
\newtheorem{prop}[defi]{Proposition}
\newtheorem{ques}[defi]{Question}
\newtheorem{lemma}[defi]{Lemma}
\newtheorem{cor}[defi]{Corollary}
\newtheorem{thm}[defi]{Theorem}
\newtheorem{expl}[defi]{Example} 
\newtheorem*{mainthm}{Main Theorem}

\parindent0pt

\title[]{Local structure of tame symmetric algebras of period four$^*$ \footnote{*This article was partially 
financed from the research grant no. 2023/51/D/ST1/01214 of the Polish National Science Center}} 

\author[K. Erdmann]{Karin Erdmann}
\address[Karin Erdmann]{Mathematical Institute, University of Oxford, UK}
\email{erdmann@maths.ox.ac.uk}

\author[A. Hajduk]{Adam Hajduk}
\address[Adam Hajduk]{Faculty of Mathematics and Computer Science, 
Nicolaus Copernicus University, Chopina 12/18, 87-100 Torun, Poland}
\email{ahajduk@mat.umk.pl} 

\author[A. Skowyrski]{Adam Skowyrski}
\address[Adam Skowyrski]{Faculty of Mathematics and Computer Science, 
Nicolaus Copernicus University, Chopina 12/18, 87-100 Torun, Poland}
\email{skowyr@mat.umk.pl} 


\begin{abstract} In this paper we study the structure of Gabriel quivers of tame symmetric algebras of 
period four. More precisely, we focus on algebras having Gabriel quiver {\it biregular}, i.e. the numbers 
of arrows starting and ending at any vertex are equal, and do not exceed $2$. We describe the local structure 
of biregular Gabriel quivers of tame symmetric algebras of period four, including certain idempotent 
algebras. The main result of this paper shows that, in fact, these Gabriel quivers have local structure 
exactly as Gabriel quivers of so called {\it weighted surface algebras} (see \cite{WSA,WSA-GV}), which 
partially extends characterization presented in \cite{AGQT}. \end{abstract}

\maketitle

\section{Introduction}\label{sec:1} 

Throughout this paper by an algebra we mean an indecomposable basic finite-dimensional $K$-algebra over an 
algebraically closed field $K$. For an algebra $\La$, we denote by $\mod\La$ the category of all finitely 
generated (right) $\La$-modules. \smallskip 

From the remarkable Tame and Wild Theorem \cite{TW} (see also \cite{CB}) every algebra lies in one of 
two disjoint classes. The first class consists of so called {\it tame algebras} for which the indecomposable 
modules occur in each dimension $d$ in a finite number of discrete and a finite number of one-parameter 
families. The second class contains the {\it wild algebras} whose representation theory comprises the 
representation theories of all algebras. \medskip 

A prominent role in the representation theory is played by the {\it selfinjective} algebras, which are algebras 
$\La$ such that $\La$ is injective in $\mod\La$, or equivalently, every projective module in $\mod\La$ 
is injective. Among the selfinjective algebras we distinguish two important classes: symmetric algebras and 
periodic algebras. \smallskip 

Recall that an algebra $\La$ is said to be symmetric if there exists an associative non-degenerate symmetric 
$K$-bilinear form on $\La$, and any such an algebra is selfinjective, due to well known characterizations; see 
\cite{SkY}. There are many classical examples of symmetric algebras, for instance, blocks of finite-dimensional 
group algebras \cite{E1} or Hecke algebras associated to Coxeter groups. Any algebra $\La$ is a quotient of its 
trivial extension $T(\La)$, which is a symmetric algebra.\smallskip 

For a module $M$ in $\mod\La$, we denote by $\Omega_\La(M)$ its syzygy, that is the kernel of a minimal projective 
cover of $M$ in $\mod\La$. A module $M$ in $\mod\La$ is called {\it periodic} if $\Omega_\La^n(M)\simeq M$, for 
some $n\geqslant 1$, and the minimal such a number is called the {\it period} of $M$. An important class of 
selfinjective algebras consists of the {\it periodic} algebras $\La$, for which $\La$ is a periodic module in 
$\mod\La^e$, where $\La^e = \La\otimes_K \La$ is the enveloping algebra of $\La$ (this is equivalent to say that 
$\La$ is periodic as an $\La$-$\La$-bimodule). Every periodic algebra $\La$ has periodic module category, that is, 
all (nonprojective) modules in $\mod\La$ are periodic (with period dividing the period of $\La$ 
\cite[see Theorem IV.11.19]{SkY}). Finding and possibly classifying periodic algebras is an important problem, as they 
appear in many places, revealing connections with group theory, topology, singularity theory, cluster algebras and 
algebraic combinatorics (see the survey \cite{periodic1}). \smallskip 

We only mention that any periodic algebra $\La$ is selfinjective, due to the results of \cite{GSnS}. Indeed, if $\La$ 
is $d$-periodic, then any simple $\La$-module is $\Omega_\La$-periodic with period dividing $d$. This is equivalent 
to say that $\La$ is so-called $d${\it -twisted periodic} algebra, which means that $\Omega_{\La^e}^d(\Lambda) \cong \ _1\La_\sigma$, 
in $\mod\La^e$, for some automorphism $\sigma$ of $\La$. Both these conditions imply that $\La$ is selfinjective, by 
\cite[Theorem 1.4]{GSnS}. The inverse question, whether any $d$-twisted periodic algebra is also periodic is a 
longstanding  conjecture (see \cite{periodic2}). \medskip 

We are concerned with the classification of all tame symmetric periodic algebras, with a special focus on the 
case of $4$-periodic algebras, which is motivated by a conjecture that all tame symmetric periodic algebras of 
non-polynomial growth are of period $4$ \cite[Problem]{WSA}. Any tame symmetric periodic algebra will be called TSP4, 
for short. Such algebras appeared naturally in the study of blocks of group algebras with generalized quaternion 
defect groups. Using known properties of these blocks, Erdmann introduced and investigated in \cite{q1,q2,q3,E1} 
the algebras of quaternion type, as  the (indecomposable) representation-infinite tame symmetric algebras with 
non-singular Cartan matrix for which every indecomposable non-projective module is periodic of period dividing $4$. 
The most recent extension concerns so-called {\it generalized quaternion algebras} \cite{AGQT}, or simply GQT algebras, 
which are by definition, the (indecomposable) tame symmetric algebras of infinite representation type with all simple 
modules of period (exactly) $4$. Note that any representation-infinite algebra of quaternion type, or more generally, 
TSP4 algebra, is a GQT algebra (see \cite{E-note}). \medskip 

In \cite{Du} Dugas proved that every representation-finite self-injective algebra, without simple blocks, is a periodic algebra. 
We note that the symmetric periodic algebras of finite representation type and representation-infinite of polynomial growth 
are classified, and each class admits a similar description as a socle deformation of an algebra of invariants of the trivial 
extension algebra of a tilted algebra (of Dynkin or tubular type; see \cite{Sk,BES}). \medskip 

Recently, motivated by cluster theory, Erdmann and Skowro\'nski defined large class of symmetric periodic algebras of 
period $4$ associated to triangulations of real compact surfaces, called \emph{weighted surface algebras} (see 
\cite{WSA,WSA-GV,WSA-corr}). These are algebras of the form 
$\Lambda=\Lambda(S,\overrightarrow{\mathcal{T}},m_\bullet,c_\bullet,b_\bullet)$ 
associated to a (directed) triangulated surface $(S,\overrightarrow{\mathcal{T}})$ and some additional data encoded 
in functions $m_\bullet$, $c_\bullet$, and $b_\bullet$. More specifically, starting from $(S,\overrightarrow{\mathcal{T}})$ 
they first construct a so-called triangulation quiver $(Q,f)$ uniquely determined by $(S,\overrightarrow{\mathcal{T}})$, 
where $Q$ is a $2$-regular quiver and $f$ is a certain permutation of arrows in $Q$ (of order $3$), and then define $\Lambda$ 
as a quotient $\Lambda=KQ/I$ of the path algebra $KQ$ of $Q$ by an ideal $I$ depending on $(Q,f,m_\bullet,c_\bullet,b_\bullet)$. 
Recall that a vertex in $Q$ is $p$-regular if there are exactly $p$ arrows starting  and $p$ arrows ending at this 
vertex, and $Q$ is called $p$-regular, provided that all of its vertices are $p$-regular. Note also that the quiver 
$Q$ is a glueing of a finite number of the following three types of {\it blocks} 
$$\xymatrix@R=0.2cm{\\ I: } \qquad \xymatrix@C=0.3cm@R=0.2cm{\\ \ar@(lu, ld)[]\circ} \qquad \qquad 
\xymatrix@R=0.2cm{\\ II: } \qquad \xymatrix@C=0.6cm@R=0.2cm{&\\ \ar@(lu, ld)[]\bullet\ar@<.35ex>[r]&\ar@<.35ex>[l]\circ} \qquad \qquad  
\xymatrix@R=0.2cm{\\ III: \\} \ \xymatrix@C=0.3cm@R=0.2cm{&\circ\ar[ld]&\\ \circ\ar[rr]&&\circ\ar[lu]}$$ 
corresponding to boundary edges, self-folded triangles, and triangles of the triangulation 
$\overrightarrow{\cT}$, respectively (for a definition of a block, see the end of Section 
\ref{sec:2}). \smallskip 

By the results of \cite{WSA,WSA-GV}, the weighted triangulation algebras are TSP4 algebras, with some minor exceptions. 
Actually, they exhaust all TSP4 algebras in some cases, as due to the results of Erdmann and Skowro\'nski from \cite{AGQT}, 
an algebra $\La$ with $2$-regular Gabriel quiver having at least three vertices is a TSP4 algebra (if and only if $\La$ 
is a GQT algebra) if and only if $\La$ is a weighted surface algebra 
$\Lambda(S,\overrightarrow{\mathcal{T}},m_\bullet,c_\bullet,b_\bullet)$ different from a singular tetrahedral algebra, 
or else it is the higher tetrahedral algebra $\Lambda(m,\lambda)$, $m\geqslant 2$, $\lambda\in K^*$. For details we refer to 
\cite{AGQT}; see also \cite{HTA}. Note that the Gabriel quiver $Q_\La$ of a weighted surface algebra $\La=KQ/I$ is not 
necesarily $2$-regular, and this is the case if and only if $Q_\Lambda=Q$. In general, $Q_\Lambda$ may be smaller than $Q$, and the 
arrows in $Q$ which are not visible in $Q_\La$, are called virtual. Nonetheless, the Gabriel quiver of any weighted surface 
algebra is {\it biregular}, that is, any vertex of $Q$ is either $1$-regular or $2$-regular. Indeed, any virtual arrow in 
$Q$ is either a loop in a block of type II, or it forms a $2$-cycle coming from glueing of two blocks of type III. As a 
result, if $\alpha$ is a virtual arrow in $Q$, then $Q$ admits one of the following two subquivers 
$$
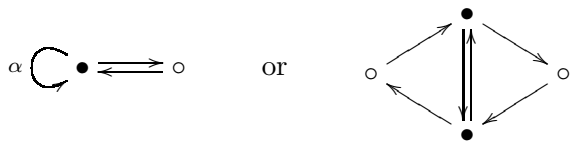
\xymatrix@R=0.4cm{&\\ \ar@(lu, ld)[]_{\alpha}\bullet\ar@<.35ex>[r]&\ar@<.35ex>[l]\circ}\qquad 
\xymatrix@R=0.4cm{\\ \mbox{or}\\} \qquad 
\xymatrix@R=0.4cm{&\bullet \ar[rd]\ar@<-0.3ex>[dd]& \\ \circ \ar[ru] &  & \circ \ar[ld]\\ & \bullet\ar[lu]\ar@<-0.3ex>[uu] & }$$ 
where in the second case $\alpha$ is one of the two arrows forming a $2$-cycle. Consequently, then the Gabriel quiver 
$Q_\La$ contains one of the following blocks 
$$
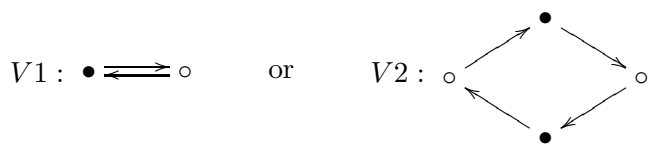
\xymatrix@R=0.4cm{\\ V1: \\ }\xymatrix@R=0.4cm{&\\ \bullet\ar@<.35ex>[r]&\ar@<.35ex>[l]\circ}\qquad 
\xymatrix@R=0.4cm{\\ \mbox{or}\\} \qquad 
\xymatrix@R=0.4cm{\\ V2: \\ }\xymatrix@R=0.4cm{&\bullet \ar[rd]& \\ \circ \ar[ru] &  & \circ \ar[ld]\\ & \bullet\ar[lu] & }$$ 
where the black vertices are $1$-regular in $Q_\La$. Using the previous blocks I-III together with V1 and V2, we can 
construct all Gabriel quivers of weighted surface algebras, including the quiver of a spherical algebra \cite[see 
Section 5]{HSS}, which is obtained by a glueing of two blocks V2. There is a nice analogy between the biregular case 
and $2$-regular case, since we also have an exceptional algebra, called higher spherical algebra \cite{HSA}, which 
is a TSP4 algebra given by the same Gabriel quiver as the spherical algebra (biregular), but it is not isomorphic 
to a weighted surface algebra. This suggest that the first natural step after \cite{AGQT} is to classify all TSP4 
(or GQT) algebras having biregular Gabriel quivers, and prove that these are either weighted surface algebras (except 
singular spherical algebras) or higher spherical algebras. \medskip 

In this paper, we focus  on describing the local structure of Gabriel quivers of TSP4 algebras, which 
is already a challenge. Our main result shows that indeed, if $\Lambda$ is a GQT algebra with biregular Gabriel 
quiver, then all $1$-regular vertices belong to blocks of types V1 or V2. Namely, we have the following theorem, 
which states the main result of this article. \smallskip 

\begin{mainthm} Let $\La$ be a GQT algebra with biregular Gabriel quiver $Q=Q_\La$. Then every $1$-regular vertex 
$i$ in $Q$ is contained in a block of the form 
$$
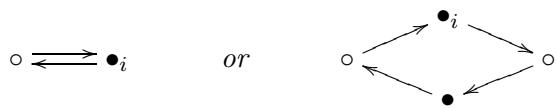
\xymatrix@R=.7ex{&&& &\bullet_i\ar[rd]& \\ 
\circ\ar@<+.4ex>[r]&\bullet_i\ar@<+.4ex>[l]&\mbox{ or }& \circ\ar[ru] && \circ\ar[ld]  \\ 
&&& &\ar[lu]\bullet&}$$ 
\end{mainthm} 

The paper is organized as follows. In the next section we recall basic notions needed in this 
article. Section \ref{sec:3} summarizes  the most important results from 
\cite{EHS} on tame symmetric algebras 
of period four used further. In Section \ref{sec:4} we discuss some preparatory results concerning $1$-regular 
vertices. Section \ref{sec:4+} presents one crucial proposition describing the local structure of 
Gabriel quiver around $1$-vertices lying in triangles, which is the most demanding result of this paper. 
In the last section we give a proof of the Main Theorem. \medskip 

For the neccessary background in the representation theory of algebras we refer the reader to books \cite{ASS,SkY}. \bigskip  

\section{Basic notions}\label{sec:2} 

By a quiver we mean a quadruple $Q=(Q_0,Q_1,s,t)$, where $Q_0$ is a finite set of vertices, $Q_1$ a finite set of 
arrows and $s,t:Q_1\to Q_0$ are functions assigning to every arrow $\alpha$ its source $s(\alpha)$ and its target 
$t(\alpha)$. For a quiver $Q$, we denote by $KQ$ the {\it path algebra} of $Q$, whose $K$-basis is given by all 
paths of length $\geqslant 0$ in $Q$. Recall that the Jacobson radical of $KQ$ is the ideal $R_Q$ of $Q$ generated 
by all paths of length $\geqslant 1$, and ideal $I$ of $KQ$ is called {\it admissible}, provided that 
$R_Q^m\subseteq I \subseteq R_Q^2$, for some $m\geqslant 2$. Note that the trivial paths $\ve_i$ (of length $0$) 
at vertices $i\in Q_0$ form a complete set of pairwise orthogonal primitive idempotents of the path algebra $KQ$, 
and their sum is the  unit of $KQ$. \smallskip 

If $Q$ is a quiver and $I$ is an admissible ideal $I$ of $KQ$, then $(Q,I)$ is said to be a {\it bound quiver}, and 
the associated algebra $KQ/I$ is called a {\it bound quiver algebra}. It is well-known that any algebra over an 
algebraically closed field is Morita equivalent to a bound quiver algebra, and by a presentation of an algebra 
$\Lambda$ we mean a particular isomorphism $\Lambda\cong KQ/I$, for some bound quiver $(Q,I)$. In this case, the 
cosets $e_i=\varepsilon_i+I\in \La$ form a complete set of primitive idempotents of $\La$ and $\sum_{i\in Q_0}e_i$ 
is the unit of $\Lambda$. \smallskip 

A {\it relation} in the path algebra $KQ$ is any $K$-linear combination of the form 
$$\sum_{i=1}^r \lambda_i w_i,\leqno{(1)}$$ 
where all $\lambda_i\in K$ are non-zero and the $w_i$ are pairwise different paths of of length $\geqslant 2$ with common 
source and target. It is known that an ideal $I$ of $KQ$ is admissible if and only if $I$ is generated by a finite number 
of relations $\rho_1,\dots,\rho_m$. Moreover, we may choose such relations $\rho_1,\dots,\rho_m$ to be minimal (i.e. each 
$\rho_i$ is not a combination of relations from $I$). For a bound quiver algebra $A=KQ/I$, given the set of (minimal) 
relations $\rho_1,\dots,\rho_m$ generating $I$, we have the (minimal) equalities $\rho_1=0,\dots,\rho_m=0$, called 
{\it minimal relations}. \smallskip 

Let $\rho$ be a relation of the form $(1)$ and $w$ a path in $Q$. We write $w\prec \rho$, if $w$ is one of the summands 
of $\rho$, i.e. $w=w_i$, for some $i\in\{1,\dots,r\}$. Moreover, if $w$ is a path in $Q$, we will use notation $w\prec I$, 
if $w\prec \rho_i$, for some $i\in\{1,\dots,m\}$, where $\rho_1,\dots,\rho_m$ are fixed minimal relations generating $I$. 
\medskip 

If $\Lambda$ is an algebra with a given presentation $\Lambda=KQ/I$, the modules $P_i=e_i\La$, for $i\in Q_0$, form a 
complete set of all pairwise non-isomorphic indecomposable projective modules in $\mod\La$, and modules the $I_i=D(\La e_i)$, 
for $i\in Q_0$, form a complete set of all pairwise non-isomorphic indecomposable injective modules in $\mod\La$. We 
denote by $S_i$, $i\in Q_0$, the associated simple module $S_i=P_i/\rad P_i\cong \soc I_i$. 

In this paper we work with algebras $\La$ which are symmetric, that is, there is a non-degenerate (associative) 
symmetric $K$-bilinear form $\La\times\La\to K$. It follows that $P_i\simeq I_i$, for any $i\in Q_0$, hence in 
particular, $\Lambda$ is then a self-injective algebra. We also assume $Q$ is connected, i.e. $\La$ is indecomposable 
as an algebra. For a vertex $i\in Q$, we denote by $p_i$ the dimension vector of $P_i$, and by $s_i$ the dimension 
vector of $S_i$. Let also $J$ denote the Jacobson radical $J=J_\La=R_Q+I$ of $\La$. \smallskip 

For $i\in Q_0$, we let $i^-$ be the set of arrows ending at $i$, and $i^+$ the set of arrows starting at $i$. In this 
paper, the sizes $|i^-|$ and $i^+|$ are at most $2$. A quiver $Q$ is said to be $2${\it -regular} if $|i^-|=|i^+|=2$, and 
{\it biserial} if $1\leq |i^-|, |i^+|\leq 2$. We say that $i\in Q_0$ is a {\it regular} vertex ($1$- or $2$-regular), 
provided $|i^-|=|i^+|$ (and the size is equal $1$ or $2$, respectively). Otherwise, we call $i$ a {\it non-regular} 
vertex. In other words, $Q$ is biserial and regular if and only if $Q_0$ consists only of $1$- or $2$-regular vertices. 
In this case, we will call $Q$ a {\it biregular} quiver, for short. \medskip 

All algebras condsidered in this article are assumed to be tame. Let us point out a few consequences of this assumption. 
First, for any two vertices $i,j\in Q_0$, there are at most two arrows $i\to j$ in $Q_1$, since otherwise, $Q$ 
has a Kronecker subquiver $\Delta$ with three arrows, and the path algebra $K\Delta$ is a wild hereditary quotient 
algebra of $\Lambda$, and  $\Lambda$ is a wild algebra. \smallskip 

Further, it is clear that $Q$ cannot have a subquiver $\Delta$ of the form 
$$K_2^+= \ \xymatrix@C=0.5cm{\circ &\ar@<+0.4ex>[l]\ar@<-0.4ex>[l]\circ\ar[r]&\circ} \qquad \mbox{or} \qquad 
K_2^-= \ \xymatrix@C=0.5cm{\circ \ar@<+0.4ex>[r]\ar@<-0.4ex>[r]&\circ&\ar[l]\circ} $$ 
because $K\Delta$ would be again a wild (hereditary) quotient algebra of $\Lambda$. It follows that 
for any two vertices $i,j$ connected by double arrows $\alpha,\ba$, there are no loops at $i,j$. Indeed, 
if there is a loop at $i$ or $j$ then there exists a quotient algebra $C$ of $B=\La/J^3$ and a Galois covering 
$\wt{C}\to C$ such that $\wt{C}$ contains a full convex subcategory of one of the forms 
$$\xymatrix@C=0.5cm{\circ &\ar@<+0.4ex>[l]^{\alpha}\ar@<-0.4ex>[l]_{\ba}\circ\ar[r]&\circ} \qquad \mbox{or} \qquad 
\xymatrix@C=0.5cm{\circ \ar@<+0.4ex>[r]^{\ba}\ar@<-0.4ex>[r]_{\alpha}&\circ&\ar[l]\circ} $$ 
which is isomorphic to a wild hereditary algebra $K\Delta$, where $\Delta$ is of type $K_2^\pm$. Then using 
\cite[Theorem]{DS2} and \cite[Proposition 2]{DS1} we conclude that $\wt{C}$ and $C$ are wild. Since $C$ is a quotient 
of $B$ and $\Lambda$, we obtain that $\La$ is wild. \smallskip 

Using analogous argumentation, one can prove that for any vertex $i\in Q_0$, there is at most one loop at $i$, 
if $Q$ has more than one vertex. In fact, if $Q$ admits two loops at vertex $i$, then appropriate Galois covering 
contains an infinite double-line 
$$\xymatrix@C=0.5cm{\dots \ar@<+0.4ex>[r]\ar@<-0.4ex>[r] & \circ 
\ar@<+0.4ex>[r]\ar@<-0.4ex>[r]&\circ\ar@<+0.4ex>[r]\ar@<-0.4ex>[r] & \dots}$$ 
being a 'resolution' of the two loops in $Q$. It follows that $Q$ has exactly one vertex, because otherwise, the 
above line can be extended to a wild subcategory of type $K_2^\pm$. \smallskip 

We say that $\Delta$ is a subquiver {\it of type} $K_2^*$ if $\Delta$ is on of the following two quivers 
$$ \xymatrix@C=0.5cm{\circ \ar@<+0.4ex>[r]^{\ba}\ar@<-0.4ex>[r]_{\alpha}&\circ\ar[r]&\circ} \ \mbox{ or } \ 
\xymatrix@C=0.5cm{\circ \ar[r]&\circ \ar@<+0.4ex>[r]^{\ba}\ar@<-0.4ex>[r]_{\alpha}&\circ}$$
and both $\alpha\beta, \bar{\alpha}\beta\nprec I$ (or both $\beta\alpha,\beta\bar{\alpha}\nprec I$). Note that 
if $\La$ is tame, its Gabriel quiver $Q=Q_\La$ does not contain a subquiver of type $K_2^*$. To see this, let 
$\Delta$ be a subquiver in $Q$ of type $K_2^*$. Then  consider the factor algebra $C=KQ/J_C$ of $B=\La/J^3$, 
where $J_C$ is generated by all paths of length $2$ in $Q$ except $\alpha\beta$ and $\bar{\alpha}\beta$, in 
the first case, and $\beta\alpha,\beta\ba$, in the second. As before, $C$ admits a Galois covering $\wt{C}\to C$ 
containing a full convex subcategory 
$$\xymatrix@C=0.5cm{\circ \ar@<+0.4ex>[r]^{\ba}\ar@<-0.4ex>[r]_{\alpha}&\circ\ar[r]&\circ} \ \mbox{ or } \ 
\xymatrix@C=0.5cm{\circ \ar[r]&\circ \ar@<+0.4ex>[r]^{\ba}\ar@<-0.4ex>[r]_{\alpha}&\circ}$$ 
which is isomorphic to a wild hereditary algebra $K\Delta$. \smallskip 

The above subquivers can be viewed as the smallest wild 'pieces' that may occur in $Q$. In this paper we will 
frequently use also other wild subquivers, or more precisely, induced bound quivers (or algebras), which are 
reffered to as the {\it wild one-relation algebras}. They were introduced by Ringel in \cite{Rin}, where the 
list of all 37 wild-one relation algebras was also presented \cite[see 1.5. Theorem 2]{Rin}. Actually, the list 
contains only underlying graphs (with unique relation marked by circle and square), and whenewer we use the term 
wild one-relation algebra, we mean an algebra such that its Gabriel quiver has underlying graph being one of 
the shapes listed in \cite{Rin}. We omit marking the unique relation, because it is always clear from the context 
what is its position. \smallskip 

We will use abbreviation RiN to indicate the wild one-relation algebra, whose number on the 
Ringel's list is $N$. Note that the list in \cite{Rin} uses three sorts of numerations to denote the graphs, 
and by a number on this list, we mean its number in the Roman numeration (one of the three conventions). For 
hereditary algebras of Euclidean (or wild) type, we use classical notation $\wt{A}_n,\wt{D}_n,\wt{E}_6,\wt{E}_7,\wt{E}_8$ 
(or $\wt{\wt{D}}_n,\wt{\wt{E}}_6,\wt{\wt{E}}_7,\wt{\wt{E}}_8$). The biggest usage of this notation will be visible 
in two crucial results of this paper, that is, the Main Theorem and Proposition \ref{prop:4.2}. \smallskip 

Now, let us only mention one simplification which is equally frequently used during the sequel, where 
many proofs contain wordings like "the algebra admits the following wild-one relation algebra ... in covering". 
By this we mean a precise statement in the language of covering theory, namely, an algebra $A$ is said to 
be a subcategory in covering of $\La$, provided that the algebra $B=\La/J^d$ (for some $d\geqslant 3$) 
has a quotient algebra $C$, which admits a Galois covering $\tilde{C}\to\tilde{C}/G=C$, with finitely generated 
free group $G$, such that $\tilde{C}$ contains a full subcategory isomorphic to $A$. 

Anytime we use this abbreviation, appropriate arguments can be verified, but we do not elaborate on this, to 
avoid making longer enough long proofs. One simple instance of an argumentation of this type was shown above 
in the case of subquivers of type $K_2^*$, where the essential results needed from covering theory are cited. 
For more background on covering techniques we refer to cited papers \cite{DS1,DS2}. \medskip 

We will use the following notation and convention for arrows: we write $\alpha, \ba$ for the arrows starting at vertex 
$i$, with the convention that $\ba$ does not exist in case $|i^+|=1$. Similarly we write $\gamma, \gamma^*$ for the 
arrows ending at some vertex $i$, where again $\gamma^*$ may not exist. \smallskip 

Then $Q$ has a subquiver 
$$
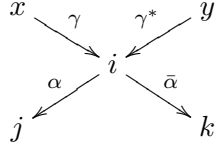
\xymatrix@R=0.3cm{x\ar[rd]^{\gamma}&&y\ar[ld]_{\gamma^*}\\&i\ar[ld]_{\alpha}\ar[rd]^{\ba}&\\j&&k}$$ 

We note that for a module $M$ in $\mod\La$, its {\it syzygy} $\Omega(M)$ is a kernel of any projective cover of $M$ 
in $\mod\La$. A module $M$ in $\mod\La$ is called {\it periodic} if $\Omega^d(M)\simeq M$, for some $n\geqslant 1$, 
and the smallest such a number is {\it period} of $M$. The notion of {\it inverse syzygy} $\Omega^{-1}(M)$ for a module 
$M$ in $\mod \La$ is defined similarly using injective envelopes. \smallskip 

Let $i\in Q_0$. Recall that there are canonical isomorphisms $\Omega(S_i)=\rad P_i=\alpha\La+\ba\La$ and 
$\Omega^{-1}(S_i)\cong (\gamma,\gamma^*)\Lambda\subset P_x\oplus P_y$ \cite[see Lemma 4.1]{AGQT}. In particular, then 
the module $P_i^+=P_j\oplus P_k$ is a projective cover of $\Omega(S_i)$ and the module $P_i^-=P_x\oplus P_y$ is an 
injective envelope of $\Omega^{-1}(S_i)$ ($\La$ is symmetric). As a result, if $S_i$ is a periodic module of period 
$4$, then $\Omega^2(S_i)\simeq\Omega^{-2}(S_i)$, and hence there is an exact sequence in $\mod\La$ of the form 
$$0\to S_i\to P_i \stackrel{d_3}\to  P_i^- \stackrel{d_2}\to  P_i^+ \stackrel{d_1}\to P_i \to S_i\to 0 \leqno{(*)}$$ 
with $\Img d_k\simeq\Omega^k(S_i)$, for $k\in\{1,2,3\}$. By our convention, $P_y$ or $P_k$ may not exist. Moreover, 
we denote by $p_i^+$ (respectively, $p_i^+$) the dimension vector of $P_i^+$ (respectively, of $P_i^-$). Using the 
above sequence, one easily gets that $p_i^+=p_i^-$. This dimension vector will be denoted by $\hat{p}_i$ (see also 
Lemma \ref{lem:3.1}). \smallskip 

We may assume that $d_1(x, y) : = \alpha x + \ba y$, since the induced epimorphism 
$(\alpha \ \ba):P_j\oplus P_k \to \Omega(S_i)=\alpha\La+\ba\La$ is a projective cover of $\Omega(S_i)$ in $\mod\La$. 
Adjusting arrows $\gamma$ or $\gamma^*$ (including impact on generators of $I$), we can fix $d_3(e_i) = (\gamma, \gamma^*)$, 
for some choice of the arrows $\gamma, \gamma^*$ ending at $i$ \cite[see Proposition 4.3]{AGQT}. \smallskip

The kernel of $d_1$ is then $\Omega^2(S_i)=\Img d_2$, and it has at most two minimal generators. They are images of 
idempotents $e_x\in P_x=e_x\La$ and $e_y\in P_y$ via $d_2:P_i^-\to P_i^+$. We denote them as $\vf$ and $\psi$, 
respectively, and then $\vf,\psi\in P_j\oplus P_k$, so we can write 
$$\vf = d_2(e_x,0) = (\vf_{jx}, \ \vf_{kx}) \ \ \mbox{and} \ \ \psi = d_2(0,e_y) = (\psi_{jy}, \ \psi_{ky}),$$
where $\vf_{ab}$ and $\psi_{ab}$ belong to $e_a\La e_b$ for all components of $\vf$ and $\psi$. \smallskip 

We recall that any homomorphism $d:P_x\oplus P_y\to P_j\oplus P_k$ in $\mod\Lambda$ can be represented in the matrix form 
$$M={m_{jx} \ m_{jy}\choose m_{kx} \ m_{ky}},$$ 
where $m_{ab}$ is a homomorphism $P_b\to P_a$ in $\mod\La$, identified with an element $m_{ab}\in e_a\La e_b$. In 
this way, $d$ becomes multiplication by $M$, i.e. $d(u)=M\cdot u$, for any $u\in P_i^-$ (here we use column notation 
for vectors in $P_i^-$ and $P_i^+$). \smallskip

Continuing with the generators of $\Omega^2(S_i)$, let $M_i$ be the matrix with columns being the components of $\vf$ and $\psi$, 
that is $d_2$ is given by matrix 
$$M_i={\vf_{jx} \ \psi_{jy} \choose \vf_{kx} \ \psi_{ky}}.$$ 
Rewriting compositions $d_1d_2=0$ and $d_2d_3=0$ in matrix form, we get identities 
$$(\alpha \ \ba)\cdot M_i = 0\mbox{ and }M_i\cdot {\gamma \choose \gamma^*} =0,$$
for some choice of arrows $\gamma, \gamma^*$ ending at $i$ (cf. \cite[Proposition 4.3]{AGQT}). \medskip 

We end this section with an auxiliary definition of a {\it block}, which appears in a natural way in the study of 
TSP4 algebras (see Introduction). Namely, by a {\it block}, we mean a subquiver $\Gamma=(\Gamma_0,\Gamma_1)$ of the 
Gabriel quiver $Q=Q_{\La}$, whose vertex set $\Gamma_0$ is a disjoint union $\Gamma_0=B\cup W$, such that every 
arrow $\alpha:i\to j$ with $i\in B$ or $j\in B$ belongs to $\Gamma_1$. Vertices in $B$ are usually marked by 
$\bullet$ (and called {\it black}), while vertices in $W$ are marked by $\circ$ (and called {\it white}, or 
{\it outlets}). \smallskip 

By definition, if $\Gamma$ is a block in $Q$ with $\Gamma_0=B\cup W$, then $Q$ consists of arrows in $\Gamma_1$, 
and all other arrows, which either connect two vertices in $Q_0\setminus\Gamma_0$, or a vertex not in $\Gamma_0$ 
with an outlet in $W$. In this way, $Q$ can be viewed as a glueing of $\Gamma$ with the rest part of the quiver 
via outlets in $W$. For a precise definition of a glueing, we refer to \cite[Section 3]{SS}. 

\bigskip 
 
\section{Tame symmetric periodic algebras}\label{sec:3} 

This section is devoted to present the main results on tame symmetric algebras of period four needed in further 
investigations. In what follows, we collect the most important results from \cite{EHS}. \smallskip 

We start with two observations concerning infinite type \cite[see Lemmas 2.1 and 2.4]{EHS}. For 
a vector $x\in\bN^n$, we denote by $|x|$ the sum $x_1+\dots+x_n$ of its coordinates. 

\begin{lemma}\label{lem:3.1} If $\La$ is of infinite type, then $|\hat{p}_i|>|p_i|$, for any vertex $i\in Q_0$. \end{lemma} 

\begin{lemma}\label{lem:3.2} If $\La$ is of infinite type, then there is no arrow $\alpha: i\to j$ with $i^+ = \{ \alpha\} = j^-$. 
\end{lemma} 

The following result gives a useful tool for constructing triangles in $Q$ induced from relations defining 
$\La=KQ/I$ \cite[see Proposition 4.1]{EHS}.

\begin{lemma}\label{lem:3.3} Assume $\alpha: i\to j$ and $\beta: j\to k$ are arrows such that $\alpha\beta \prec I$. Then there
is an arrow in $Q$ from $k$ to $i$, so that  $\alpha$ and $\beta$ are part of a triangle in $Q$. \end{lemma} 

The following result (sometimes called the {\it Triangle Lemma} \cite[Lemma 4.3]{EHS}) shows how relations 
propagate through triangles in $Q$. 

\begin{lemma}\label{lem:3.4}  Assume $Q$ contains  a triangle
        
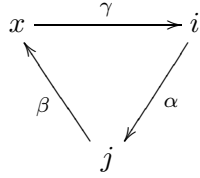
\[
 \xymatrix@R=3.pc@C=1.8pc{
%  \xymatrix@C=.8pc{
    x
    \ar[rr]^{\gamma}
    && i
    \ar@<.35ex>[ld]^{\alpha}
    \\
    & j
    \ar@<.35ex>[lu]^{\beta}
  }
\]
with $\alpha\beta\prec I$. If $\gamma$ is the unique arrow $x\to i$, then $\gamma\alpha\prec I$ and 
$\beta\gamma\prec I$. If we have double arrows $\gamma,\bar{\gamma}: x\to i$, then there is a 
$\delta\in\{\gamma,\bar{\gamma}\}$ such that $\delta\alpha\prec I$ and $\beta\delta\prec I$. \end{lemma} 

We have also the following result on the neighbours of $1$-vertices \cite[see Lemma 4.4]{EHS}. 

\begin{lemma}\label{lem:3.5} Assume $i$ is a 1-vertex which is part of a triangle 

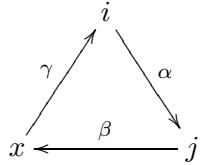
\[
%  \xymatrix@R=2pc@C=1.5pc{
%  \xymatrix@R=3.5pc@C=1.8pc{
  \xymatrix@R=3.pc@C=1.8pc{
%  \xymatrix@C=.8pc{
    & i
    \ar[rd]^{\alpha}
    \\
  x 
    \ar[ru]^{\gamma}
  %  \ar@<-.5ex>[rr]_{\eta}
    && j
   \ar[ll]_{\beta}
% \ar[ld]^{\beta}
%    \\
%    & d
 %   \ar[lu]^{\nu}
  }
\]

Then both $x$ and $j$ must be 2-vertices. \end{lemma} 

In fact, the above result does not depend on the assumptions on $Q$. Namely, for arbitrary $Q$, the same 
arguments may be adapted to show that for a $1$-vertex $i$, its neighbours satisfy $|x^-|\geqslant 2$ and 
$|j^+|\geqslant 2$. Combining this with Lemma \ref{lem:3.2}, we could conclude that in infinite type, neighbours 
of a $1$-vertex are at least $2$-vertices. \smallskip 

Finally, the last two results \cite[see Propositions 4.5 and 4.6]{EHS} describe the properties of paths of length 
$3$ involved in minimal relations. 

\begin{lemma}\label{lem:3.6} Consider a path 
$$\xymatrix{i \ar[r]^{\alpha} & k \ar[r]^{\beta} & t \ar[r]^{\gamma} & j}$$ 
such that $\alpha\beta \not\prec I$ and $\alpha$ is the unique arrow $i\to k$. If $\alpha\beta\gamma \prec I$, then 
there is an arrow $j\to i$. \end{lemma}

\begin{lemma}\label{lem:3.7} Assume $Q$ contains a square 
$$
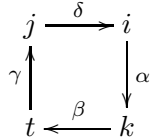
\xymatrix{j\ar[r]^{\delta} & i\ar[d]^{\alpha} \\ t\ar[u]^{\gamma} & k\ar[l]_{\beta}}$$ 
with $\alpha\beta\gamma\prec I$ but $\beta\gamma\nprec I$. If $\delta$ is the unique arrow $j\to i$, then 
$\beta\gamma\delta\prec I$. \end{lemma}  

\bigskip

\section{Preparatory results}\label{sec:4}

In this paragraph we will prove a few technical results on the local structure of the Gabriel quiver near $1$-regular 
vertices. \smallskip

Assume $\La$ is a TSP4 algebra of infinite representation type with given presentation $\La=KQ/I$ such that $Q$ 
is biregular. Due to results of \cite{AGQT}, we may assume that $Q$ admits at least one $1$-vertex, so that $Q$ has a 
subquiver $$x \stackrel{\gamma}\longrightarrow i \stackrel{\alpha}\longrightarrow j$$
and $i$ is a 1-vertex. Then from the exact sequence $(*)$ for $S_i$ we know that $p_x=p_j$. As well $\La$ is symmetric, therefore the Cartan matrix of $\La$ has a submatrix (corresponding to vertices $i, j, x$) 
$$\left(\begin{matrix} a_0 & a_1 & a_1  \cr a_1 & a_2 & a_2  \cr a_1 & a_2 & a_2 \end{matrix}\right) 
\ \leqno{(\dagger)}$$

\medskip 

\begin{lemma}\label{2triangles} The quiver $Q$ does not contain a subquiver 
$$
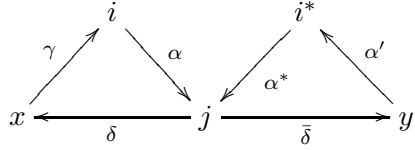
\xymatrix{&i\ar[rd]^{\alpha}&&i^*\ar[ld]^{\alpha^*}& \\ 
x\ar[ru]^{\gamma}&&j\ar[rr]_{\bar{\delta}}\ar[ll]^{\delta} && y \ar[lu]_{\alpha'}}$$ 
with both $i$ and $i^*$ being $1$-regular. 
\end{lemma}  

\begin{proof} Suppose that $Q$ admits a subquiver of the above form. Using $p^-_\bullet=p^+_\bullet$ for vertices 
$i,i^*$ and $j$, we obtain $p_x=p_j=p_y$ and $p_i+p_{i^*}=p_x+p_y=2p_j$. We denote by $p$ the dimension vector 
$p:=p_j=p_x=p_y$. \smallskip 

Now, observe that there is a commutative diagram in $\mod A$ with exact rows of the form 
$$\xymatrix@C=1cm{
0\ar[r] & S_j\ar[r]\ar[d]_{v} & P_j\ar[rr]^{\vec{\alpha \\ \alpha^*}} \ar[d]_{\vec{1\\1}} & & 
(\alpha,\alpha^*)\Lambda \cong \Omega^{-1}_{j}\ar[r]\ar[d]&0 \\ 
0\ar[r]&\Omega^2_{i}\oplus\Omega^2_{i^*} \ar[r] &P_j\oplus P_j \ar[rr]^{\vec{\alpha & 0 \\ 0 & \alpha^*}} & & 
\alpha\Lambda\oplus\alpha^*\Lambda=\Omega_{i}\oplus \Omega_{i^*} \ar[r] & 0}$$ 
where $\Omega^k_a$ denotes the $k$-th syzygy $\Omega^k(S_a)$. Clearly, $v$ is a monomorphism, but not an isomorphism, 
so its cokernel $C$ must be nonzero. On the other hand, using syzygies of $S_i$ and $S_{i^*}$, we easily get the 
following equalities of dimension vectors: 
$$[\Omega^2_i]=p_j-[\Omega_i]=p-p_i+s_i \quad\mbox{and}\quad [\Omega^2_{i^*}]=p_j-[\Omega_{i^*}]=p-p_{i^*}+s_{i^*}.$$ 
By $[X]$ we denote the dimension vector of a module $X$. Therefore, using $2p=p_i+p_{i^*}$, it follows that 
$$[C]=[\Omega^2_{i}]+[\Omega^2_{i^*}]-s_j=p-p_i+s_i+p-p_{i^*}+s_{i^*}-s_j=$$
$$=2p-(p_i+p_{i^*})+s_i+s_{i^*}-s_j=s_i+s_{i^*}-s_j,$$ 
and we obtain a contradiction, since then $[C]$ has the $j$-th coordinate negative. \end{proof} 

\medskip 

\begin{lemma}\label{lem:2vertices} There is no tame symmetric algebra $A=KQ/I$ with 
$Q=\xymatrix@R=0.6cm@C=0.45cm{x\ar@/_10pt/[rr] &&j\ar@<-.5ex>[ll]\ar@<+.4ex>[ll]}$ 
and $C_A$ of the form $C_A=\vec{a & a \\ a & a}$. \end{lemma}

\begin{proof} Suppose there is an algebra $A$ satisfying the above properties. Let $\gamma$ denote the unique arrow $x\to j$, 
and $\beta,\bar{\beta}$, the double arrows $j\to x$. \smallskip  

Resolving the double arrows in covering, we get the infinite line with of the form 
$$\xymatrix{\dots \ar[r]^{\beta} & x & j\ar[l]_{\bar{\beta}} \ar[r]^{\beta} & x & \ar[l]_{\bar{\beta}} \dots }$$ 
Therefore, since $A$ is tame, we must have a relation of the form $\gamma(c_1\beta+c_2\bar{\beta})$ modulo 
radical cube. We may replace one of the arrows, and assume that $\gamma\bar{\beta}\in J^3$. \smallskip 

Now, the module $e_xA$ has a basis consisting of monomials starting with $\gamma$, where $\gamma$ and $\beta$ 
alternate. As a result, $e_xAe_x$ has a basis of the form $\{(\gamma\beta)^k; \ 0\leq k \leq m \}$, for some 
$m\geqslant 1$ such that $(\gamma\beta)^m$ is in the socle of $A$. But then $e_xAe_j$ has a basis consisting 
of paths $(\gamma\beta)^k\gamma$, for $0\leq k \leq m-1$, so $a_{11}=\dim_K e_xAe_x> \dim_K e_xAe_j=a_{12}$, 
and the Cartan matrix does not have a required form. \end{proof} 

\medskip 

\begin{cor}\label{exept.5} There is no GQT algebra $\La=KQ/I$, whose Gabriel quiver $Q=Q_\La$ has the form:  
$$
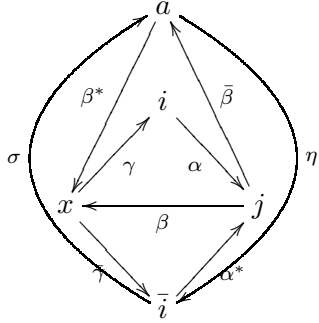
\xymatrix{& \ar[ldd]_{\beta^*} a \ar@/^50pt/[ddd]^{\eta} & \\ 
& i \ar[rd]_{\alpha} & \\ x \ar[ru]_{\gamma} \ar[rd]_{\bar{\gamma}} && j\ar[ll]^{\beta} \ar[luu]_{\bar{\beta}} \\ 
& \bar{i} \ar[ru]_{\alpha^*} \ar@/^50pt/[uuu]^{\sigma} & }$$ 
with $\gamma\alpha\prec I$. \end{cor} 

\begin{proof} Suppose $\La$ is an GQT algebra with the above quiver. First, observe that $\beta^*\gamma\nprec I$, 
by Lemma \ref{lem:3.3}, because we have no arrows $i\to a$. It follows that $\beta^*\bar{\gamma}\prec I$, 
since otherwise, we would get the following wild subcategory in the covering
$$
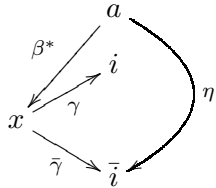
\xymatrix@R=0.3cm{& \ar[ldd]_{\beta^*} a \ar@/^30pt/[ddd]^{\eta} & \\ 
& i & \\ x \ar[ru]_{\gamma} \ar[rd]_{\bar{\gamma}} && \\ 
& \bar{i}  & }$$

Using the Triangle Lemma \ref{lem:3.4}, we infer that also $\bar{\gamma}\sigma\prec I$, since we have 
a triangle $(\beta^* \ \bar{\gamma} \ \sigma)$. \smallskip 

Now, consider the minimal relation involving $\gamma\alpha$. After possibly adjusting $\alpha$ and $\gamma$, 
we may assume that $\gamma\alpha=\bar{\gamma}z\alpha^*$, for some $z\in \La$. Consequently, any path from 
$x$ factorizes through $\bar{\gamma}$, and hence, $\bar{\gamma}\sigma$ is involved in a minimal relation 
of the form 
$$\bar{\gamma}\sigma=\gamma\alpha z_1+ \bar{\gamma}z_2=\bar{\gamma}(z\alpha^*z_1+z_2),$$ 
so changing presentation $\sigma:=\sigma-z\alpha^*z_1+z_2$, we can assume $\bar{\gamma}\sigma=0$. \smallskip 

Finally, consider the idempotent algebra $A=(e_x+e_j)\Lambda(e_x+e_j)$. Its Gabriel quiver contains arrow 
$\beta:j\to x$ induced from $Q$. Moreover, $\bar{\beta}\beta^*\nprec I$, due to Lemma \ref{lem:3.3}, since 
there are no arrows $x\to j$ in $Q$. Hence, there is another arrow $\wt{\beta}:j\to x$ in $Q_A$, which 
is induced from the path $\bar{\beta}\beta^*$. The algebra $A$ is symmetric, so we must have at least one 
arrow $x\to j$, and by Lemma \ref{lem:2vertices}, we cannot have only one arrow $x\to j$, because $i$ is 
$1$-regular, so $p_x=p_j$, and hence $C_A$ has equal columns. As a result, we have two arrows $x\to j$, which 
can be identified with paths $u,v$ in $Q$. Both of these paths start with $x$, so they are of the form 
$u=\bar{\gamma}u'$ and $v=\bar{\gamma}v'$, for some $u',v'\in\La$. But $\bar{\gamma}\sigma=0$, so both $u'$ 
and $v'$ must start with $\alpha^*$, and consequently, we obtain that $u=\bar{\gamma}\alpha^*u''$ and 
$v=\bar{\gamma}\alpha^*v''$, for $u'',v''\in e_j\La e_j=e_jAe_j$. Because $u,v$ represent arrows in $Q_A$ and 
$\bar{\gamma}\alpha^*\in J_A$, we conclude that $u''$ and $v''$ does not belong to $J_A$, so they are identified 
with units in $K$. It follows that $\lambda u+ \mu v=0$, for some $\lambda,\mu\in K$, which is impossible, since 
$u,v$ are independent arrows in $J_A/J^2_A$. \end{proof}

\medskip 

Let us finish this section with the following observation. 

\begin{lemma}\label{observ} $Q$ does not contain a subquiver of the form:  
$$
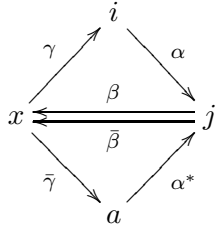
\xymatrix{&i\ar[rd]^{\alpha} & \\ 
x\ar[ru]^{\gamma} \ar[rd]_{\bar{\gamma}} && j \ar@<-0.4ex>[ll]_{\beta} \ar@<+0.4ex>[ll]^{\bar{\beta}} \\ 
&a\ar[ru]_{\alpha^*} &}$$ 
\textit{where i is a 1-vertex and }$\gamma\alpha\prec I$. 
\end{lemma} 

\begin{proof} First, by Lemma \ref{lem:3.3}, we may choose one arrow, say $\beta$, with $\alpha\beta,\beta\gamma\prec I$. 
Actually, using the exact sequence for $S_i$ (and adjusting $\beta$ if necessary), we may assume that 
$\alpha\beta=0=\beta\gamma$; see also part (1a) in the proof of \ref{prop:4.2}. In particular, it follows that 
$\bar{\beta}\gamma \nprec I$, since otherwise, we get $\bar{\beta}\gamma=u\bar{\beta}\gamma=\dots=0$, and then 
we have an arrow $\gamma:x\to i$ in $\soc(\La)$. \smallskip 

Further, because $\Lambda$ is a tame algebra, one of $\beta\bar{\gamma},\bar{\beta}\bar{\gamma}$ is involved 
in a minimal relation. Hence we conclude from \ref{lem:3.3} that also $\bar{\gamma}\alpha^*\prec I$. We 
may write (possibly adjusting $\alpha^*$): 
$$\bar{\gamma}\alpha^*=\gamma\alpha z_1,$$ 
for some $z_1\in\Lambda$, since $i$ is a $1$-vertex. On the other hand, we know that $\gamma\alpha\prec I$, so 
we can write $\gamma\alpha=\bar{\gamma}z\alpha^*$ (see (1b) in the proof of \ref{prop:4.2}), and hence we get 
$$\bar{\gamma}\alpha^*=\bar{\gamma} z \alpha^* z_1.$$ 
As a result, after replacing $\alpha^*:=\alpha^*-z\alpha^*z_1$, we can assume that $\bar{\gamma}\alpha^*=0$. \smallskip 

There must be an arrow $\bar{\alpha}^*$ in $Q$ different from $\alpha^*$ starting at $a$. Otherwise we would 
have $\bar{\gamma}J_{\La} = \bar{\gamma}\alpha^*\La = 0$ and $\bar{\gamma}$ would be in the socle and in 
$e_x\La e_a$, which is not possible for symmetric $\La$. Similarily, we have an arrow $\bar{\gamma}^*\neq\bar{\gamma}$ 
ending at $a$. \medskip 

Now, consider the idempotent algebra $R=e\Lambda e$, where $e=e_i+e_x+e_j+e_a$. Then, its Gabriel quiver 
contains all the arrows $\gamma,\bar{\gamma},\alpha,\alpha^*,\beta,\bar{\beta}$, and we have $\bar{\gamma}\alpha^*=0$ 
in $R$. Moreover, these arrows exhaust all arrows in $Q_R$ with source or target equal to $i,x$ or $j$. As a result, 
besides these arrows, we can only have (at most one) loop at $a$. Since $R$ is also tame symmetric, we can prove 
as above that there must be at least one arrow in $Q_R$ starting at $a$ which is different from $\alpha^*$. As a result, 
there is a loop $\sigma$ in $Q_R$ at vertex $a$. Additionally, $\sigma$ may be taken as the cycle at $a$ in $Q$ 
starting with $\bar{\alpha}^*$. \smallskip 

Observe also that $\bar{\gamma}\sigma\nprec I_R$. Indeed, because $\bar{\gamma}\alpha^*=0$ and $\gamma\alpha\prec I$, 
we conclude that if $\bar{\gamma}\sigma\prec I_R$, then it is involved in a minimal relation of the form 
$\bar{\gamma}\sigma=\bar{\gamma}z$, where $z\in\sigma R$. But then, by repeating substitution, we get 
$\bar{\gamma}\sigma=0$, which gives an arrow $\bar{\gamma}$ in the socle of $R$, a contradiction. \smallskip 

Next, we claim that $\sigma^2 \prec I_R$. Otherwise we have the following wild subcategory in covering 
$$\xymatrix@R=0.5cm{&&& i &&& \\ &&& x \ar[d]_{\bar{\gamma}} \ar[u]^{\gamma} &&& \\ 
i & \ar[l]_{\gamma} x \ar[r]^{\bar{\gamma}} & a & \ar[l]_{\sigma} a & 
\ar[l]_{\sigma} a \ar[r]^{\alpha^*} & j & \ar[l]_{\alpha} i}$$ 

Recall that $\gamma\alpha\prec I$ and it is involved in a minimal relation of the form 
$\gamma\alpha=\bar{\gamma} z \alpha^*$, for some $z\in \La$. It follows that there is an induced relation 
$\gamma\alpha=\bar{\gamma}\omega\alpha^*$ in $R$, and consequently, we have $e_xRe_j=\bar{\gamma}R\alpha^*$. \medskip 

Further, we will identify the basis of local algebra $R_a=e_aRe_a$. We note that there is some monomial in the 
arrows, say $X= \beta q$, which spans the socle of $e_jR$. Then $q$ ends with $\alpha$ or $\alpha^*$. It must 
end in $\alpha^*$, because otherwise the rotation $q\beta$ would be zero. But the algebra $R$ is symmetric, so the 
rotation also is non-zero in the socle. Hence $q= q'\alpha^*$ and $X=\beta q'\alpha^*$. \smallskip 

Clearly $q'$ starts with $\bar{\gamma}$, since $\beta\gamma=0$, and then 
$X= \beta \bar{\gamma}\sigma \alpha^*\ldots \alpha^*$, because $\sigma^2\prec I$. Therefore 
$X=(\beta\bar{\gamma}\sigma\alpha^*)^m$ for some $m\geq 1$. Then by rotation $(\bar{\gamma}\sigma\alpha^*\beta)^m$
spans the socle of $e_xR$ and both $(\sigma\alpha^*\beta\bar{\gamma})^m$ and $(\alpha^*\beta\bar{\gamma}\sigma)^m$ 
span the socle of $e_aR$. In particular, they are equal up to a scalar. \smallskip 

Concluding, we obtain a basis of the local algebra $e_aRe_a=e_a\Lambda e_a$ formed by the following paths: 
$$(\sigma\alpha^*\beta\bar{\gamma})^k \qquad \mbox{ and } \qquad (\alpha^*\beta\bar{\gamma}\sigma)^k,$$ 
for $0\leq k \leq m$, and $(\sigma\alpha^*\beta\bar{\gamma})^m=(\alpha^*\beta\bar{\gamma}\sigma)^m$ spans the socle 
of $R_a$. We denote by $C$ and $C_*$ the cycles $C=\sigma\alpha^*\beta\bar{\gamma}$, and 
$C_*=\alpha^*\beta\bar{\gamma}\sigma$, respectively. \medskip 

Finally, because $C^m=C_*^m$ is in the socle, all initial subpaths of $C$ and $C_*$ are independent. 
Moreover, any path from $x$ to $j$ starting with $\gamma\alpha$ factors through $\bar{\gamma}$. As a 
result, involving previous relations (and rotations of $C$ and $C_*$ in the socle), we conclude that 
$e_x Re_j$ admits a basis of the form 
$$\{(\bar{\gamma}\sigma\alpha^*\beta)^k\bar{\gamma}\sigma\alpha^*; \ 0\leq k \leq m-1\},$$ 
so in particular, we have $c_{jx}=\dim e_xRe_j=m$. On the other hand, we similarily get the following 
$m$ independent paths 
$$(\beta\bar{\gamma}\sigma\alpha^*)^k\beta \in e_j R e_x,$$ 
for $0\leq k \leq m-1$. But $\bar{\beta}$ is not involved in any relation, so it is another element 
of the basis, and we get $c_{xj}=\dim e_j Re_x = m+1> c_{jx}$, which is impossible for a symmetric 
algebra $R$. \end{proof} \medskip

\bigskip

\section{1-vertices in triangles}\label{sec:4+} 

In this section we present a crucial result, which describes the local structure of $Q$ near $1$-vertices 
in triangles (see Proposition \ref{prop:4.2}). \smallskip 

The main aim of this section is to prove the following result. 

\begin{prop}\label{prop:4.2} If $i\in Q_0$ is a $1$-vertex in a triangle, then $Q$ contains a block of the form  
$$
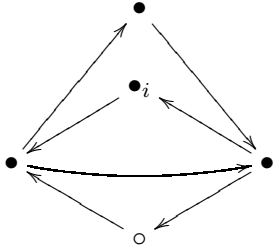
\xymatrix@C=1.2cm@R=0.6cm{&\bullet\ar[rdd]& \\ &\bullet_{i}\ar[ld]& \\ 
\bullet\ar[ruu]\ar@/_5pt/[rr]&&\bullet\ar[lu]\ar[ld] \\ &\circ\ar[lu]&}$$ \end{prop} \medskip 

Fix a GQT algebra $\La=KQ/I$ and let us start with the following observation. 

\begin{lemma}\label{lem:triangleI} Let $i$ be a $1$-vertex in a triangle 
$$
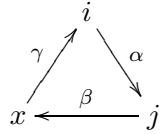
\xymatrix@C=0.5cm{&i\ar[rd]^{\alpha}&\\ x\ar[ru]^{\gamma}&&j\ar[ll]_{\beta}} $$ 
with $\gamma\alpha\nprec I$. Then $i$ is contained in the following block of $Q$: 
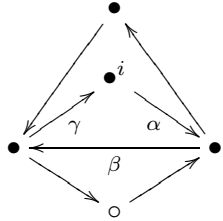
$$\xymatrix@R=0.45cm{ & \ar[ldd] \bullet & \\ & \bullet^i \ar[rd]_{\alpha} & \\ 
\bullet \ar[rd] \ar[ru]_{\gamma} && \ar[ll]^{\beta} \bullet \ar[luu] \\ 
& \circ \ar[ru] & }$$\end{lemma} 

\begin{proof} Let $i$ be a $1$-vertex in a triangle as in the statement. Note first that the identity $p_i^-=p_i^+$ 
gives $p_x=p_j$. It follows that there are no loops at $x$ and $j$. Indeed, if there is a loop at $x$, then 
$p_x^-=p_x^+$ implies that $p_i=p_j$, which gives a contradiction with Lemma \ref{lem:3.1}, since $p_j=p_i^+$. 
Similarily, a loop at $j$ gives $p_i=p_x=p_i^-$. \medskip 

Because $\gamma\alpha\nprec I$, we conclude that also $\alpha\beta\nprec I$ and $\beta\gamma\nprec I$, by Lemma 
\ref{lem:3.4} ($\gamma$ and $\alpha$ are the unique arrows between their ends). It follows from Lemma \ref{lem:3.5} 
that $x$ and $j$ are $2$-regular vertices, so $Q$ admits the following subquiver 
$$
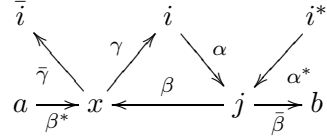
\xymatrix@R=0.6cm@C=0.55cm{\bar{i}&&i\ar[rd]^\alpha&&i^*\ar[ld]^{\alpha^*} \\ 
a\ar[r]_{\beta^*}&x\ar[lu]^{\bar{\gamma}}\ar[ru]^\gamma&&j\ar[ll]_{\beta}\ar[r]_{\bar{\beta}}&b}$$ 

First, observe that $a\neq j$ (or equivalently, $b\neq x$). If this is not the case, then we get double 
arrows $\beta,\bar{\beta}:j\to x$. If $\alpha\beta\prec I$ or $\alpha\bar{\beta}\prec I$, then by Lemma \ref{lem:3.4}, 
we get $\gamma\alpha\prec I$, which contradicts the assumption, so we obtain that $\alpha\beta\nprec I$ and 
$\alpha\bar{\beta}\nprec I$. But then we have a wild subquiver of type $K_2^*$ given by arrows $\alpha,\beta,\bar{\beta}$. 
Hence indeed, both $a\neq j$ and $b\neq x$. In particular, neither $\alpha\beta$ nor $\alpha\bar{\beta}$ are 
involved in minimal relations of $I$, by Lemma \ref{lem:3.3}. Dually, $\beta\gamma\nprec I$ and $\beta^*\gamma\nprec I$. \medskip

Now we will prove that $i^*=\bar{i}$. Suppose to the contrary that $i^*\neq i$. Then it follows from Lemma \ref{lem:3.3} 
that $\alpha^*\beta\nprec I$ and $\beta\bar{\gamma}\nprec I$. If $\alpha^*\bar{\beta}\nprec I$ or $\beta^*\bar{\gamma}\nprec I$, 
then all paths of length two passing through $j$ or $x$ are not involved in minimal relations, so we obtain a wild 
subquiver of type $\wt{\wt{\bD}}_4$. As a result, we may assume that both $\alpha^*\bar{\beta}\prec I$ and 
$\beta^*\bar{\gamma}\prec I$. In this case, we have two triangles adjacent to the triangle $(\gamma \ \alpha \ \beta)$, 
and therefore, we deduce from Proposition \ref{2triangles} that both $i^*$ and $\bar{i}$ are $2$-regular. But then we 
get a wild subcategory of type $\wt{\wt{\bD}}_6$ in the covering: 
$$\xymatrix@R=0.6cm{&& j \ar[d]_{\beta}&& i^* \ar[d]_{\alpha^*} & \\ 
\circ \ar[r] & \bar{i} & \ar[l]_{\bar{\gamma}} x \ar[r]^{\gamma} & i \ar[r]^{\alpha} & j \ar[r]^{\beta} & x }$$ 
again a contradiction. Note that $\gamma\alpha\beta\nprec I$ and $\beta\gamma\alpha\nprec I$, due to Lemma \ref{lem:3.6} 
(and it dual), since we have no loops at $x$ and $j$. Moreover, we have $\beta\gamma\alpha\beta\nprec I$, by 
\cite[Lemma 4.7]{EHS}, because this path passes through a $1$-vertex $i$. So we proved that $i^*=\bar{i}$. \medskip

Next, observe that $a=b$, since for $a\neq b$, we would have $\beta^*\gamma\alpha,\gamma\alpha\bar{\beta}\nprec I$, 
hence the following wild subcategory in the covering \cite[see also Lemma 4.7]{EHS}. 
$$\xymatrix@R=0.6cm{ && a \ar[d]_{\beta^*} && b & \\ 
b & \ar[l]_{\bar{\beta}} j \ar[r]^{\beta} & x \ar[r]^{\gamma} & i \ar[r]^{\alpha} & j \ar[u]^{\bar{\beta}} \ar[r]^{\beta} & x }$$ 
\medskip 

Consequently, we have $\bar{i}=i^*$ and $a=b$, and therefore, $Q$ has a subquiver $\Gamma$ of the form   
$$\xymatrix@R=0.45cm{ & \ar[ldd]_{\beta^*} a & \\ & i \ar[rd]_{\alpha} & \\ 
x \ar[rd]_{\bar{\gamma}} \ar[ru]_{\gamma} && \ar[ll]^{\beta} j \ar[luu]_{\bar{\beta}} \\ & \bar{i} \ar[ru]_{\alpha^*} & }$$ 

We claim that $\Gamma$ is a block in $Q$ with the unique white vertex $\bar{i}$. Indeed, the vertex $a$ must be $1$-regular, 
because otherwise, one gets a wild subcategory given as follows. 
$$\xymatrix@R=0.6cm{& a & a \ar[d]_{\beta^*} \ar[r] & \circ &\\ 
i \ar[r]^{\alpha} & j \ar[u]^{\bar{\beta}} \ar[r]^{\beta} & x \ar[r]^{\gamma} & i &}$$ 

Applying Lemma \ref{lem:3.3}, we obtain $\beta^*\bar{\gamma}\nprec I$ and $\alpha^*\bar{\beta}\nprec I$. 
Consequently, we must have $\beta\bar{\gamma}\prec I$, since in case $\beta\bar{\gamma}\nprec I$, we obtain the 
following wild subcategory of type $\wt{\wt{\bD}}_4$. 
$$\xymatrix@R=0.5cm{ &i&& \\ 
a \ar[r]^{\beta^*} & x \ar[u]^{\gamma} \ar[d]_{\bar{\gamma}} & \ar[l]_{\beta} j \ar[r]^{\bar{\beta}} & a \\ 
& \bar{i} &&  }$$ 
Therefore, we get $\beta\bar{\gamma}\prec I$, and hence, also $\bar{\gamma}\alpha^*\prec I$, by \ref{lem:3.3}. 
Thus there is a minimal relation of $I$ of the form:  
$$\bar{\gamma}\alpha^*=\gamma\alpha z,$$ 
for some $z\in\Lambda$ (after possibly adjusting $\alpha^*$). Using the exact sequence for $S_i$, we conclude 
that there is an element $m\in e_j\Lambda e_x$ such that $m\gamma=0$. \medskip 

Now, it follows that $\bar{i}$ is $2$-regular. Indeed, otherwise $\alpha^*$ is the unique arrow starting at 
$\bar{i}$, so this yields an element $m\bar{\gamma} \in e_j\Lambda e_{\bar{i}}$ lying in the (right) socle 
of $\Lambda$, because $m\bar{\gamma}\alpha^*=m\gamma\alpha z = 0$. This is impossible for a symmetric algebra. All arrows 
starting or ending at $x,j,i$ are known, hence indeed, the subquiver $\Gamma$ is a block in $Q$. \end{proof} 

Now we are ready to prove Proposition \ref{prop:4.2} extending the above lemma. \bigskip

\begin{proof}{\it (of Proposition \ref{prop:4.2})} Let $i$ be a $1$-vertex in a triangle 
$$
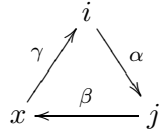
\xymatrix@C=0.5cm{&i\ar[rd]^{\alpha}&\\ x\ar[ru]^{\gamma}&&j\ar[ll]_{\beta}} $$ 
Then it is clear from Lemma \ref{lem:3.5} that $x$ and $j$ are $2$-vertices, hence there are arrows $\beta^*: a\to x$ 
and $\bar{\beta}:j\to b$ different from $\beta$ and arrows $\bar{\gamma}:x\to\bar{i}$ and $\alpha^*:i^*\to j$ with 
$\bar{i},i^*\neq i$. \smallskip 

We claim first that $a\neq j$ (equivalently, $b\neq x$). Suppose that this is not the case. Then $a=j$, $b=x$, 
and $Q$ admits a subquiver of the form 
$$
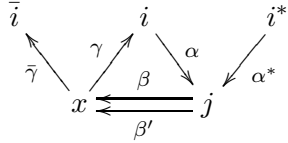
\xymatrix@R=0.6cm@C=0.45cm{\bar{i}&&i\ar[rd]^{\alpha}&&i^*\ar[ld]^{\alpha^*} \\ 
&x\ar[lu]^{\bar{\gamma}}\ar[ru]^{\gamma}&&j\ar@<-.5ex>[ll]_{\beta}\ar@<+.5ex>[ll]^{\beta'}&}$$ 
with $\beta'=\beta^*=\bar{\beta}$. Moreover, there is  no wild subcategory, hence $\alpha\beta\prec I$ or 
$\alpha\beta'\prec I$. Without loss of generality, we assume $\alpha\beta\prec I$. Then $\gamma\alpha\prec I$, 
by Lemma \ref{lem:3.4}, since there is a triangle $(\gamma \ \alpha \ \beta)$ and there are no double arrows 
between $x$ and $i$. As a result, we conclude that $i^*\neq\bar{i}$, because otherwise we obtain a subquiver 
forbidden in Lemma \ref{observ}. Now $i^*\neq\bar{i}$ implies that we have no arrows $x\to i^*$, and therefore, 
both $\alpha^*\beta\nprec I$ and $\alpha^*\beta'\nprec I$, due to Lemma \ref{lem:3.3}. But this is impossible, 
since then we get a wild subquiver of type $K_2^*$. \medskip

Consequently, we must have  $a\neq j$ (and we can assume $b\neq x$). Then $Q$ has a subquiver 
of the form $$
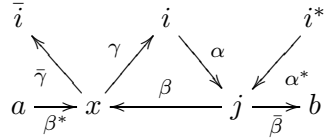
\xymatrix@R=0.6cm@C=0.55cm{\bar{i}&&i\ar[rd]^\alpha&&i^*\ar[ld]^{\alpha^*} \\ 
a\ar[r]_{\beta^*}&x\ar[lu]^{\bar{\gamma}}\ar[ru]^\gamma&&j\ar[ll]_{\beta}\ar[r]_{\bar{\beta}}&b}$$ 
Obviously, we have $\alpha\bar{\beta}\nprec I$ and $\beta^*\gamma\nprec I$, by Lemma \ref{lem:3.3}. \smallskip 

We will prove that $\bar{i}=i^*$ is a $2$-regular vertex, whereas $a=b$ is $1$-regular. Then of course, $Q$ has the 
required local shape. To do this, we will investigate several idempotent algebras of $\Lambda$. \smallskip 

We can assume that there are no loops at $x$ and $j$. Indeed, we have $p_x=p_i^-=p_i^+=p_j$, because $i$ is a 
$1$-vertex. Moreover, a loop at $x$ or $j$ implies $p_i=p_j$ or $p_i=p_x$, respectively, and in both cases we 
have $\hat{p}_i=p_i$, which contradicts Lemma \ref{lem:3.1}. Note also that $i$ is a $1$-vertex in a triangle, 
hence, we may further assume that $\gamma\alpha\prec I$ (the case $\gamma\alpha\nprec I$ is already covered in 
Lemma \ref{lem:triangleI}). \medskip  

{\bf (1)} First, let us note some properties of relations in $\Lambda$ which will be used further. Since there are no 
double arrows from $j$ to $x$ in $Q$ ($a\neq j$), we deduce from Lemma \ref{lem:3.3} that $\gamma\alpha\prec I$ 
implies $\alpha\beta\prec I$ and $\beta\gamma\prec I$. \smallskip 

(1a) We may assume that $\alpha\beta=\beta\gamma=0$ in 
$\Lambda$. To see this, consider the exact sequence for $S_i$: 
$$0\to \gamma\La \to P_x \to P_j \to \alpha \La \to 0.$$ 
We write down the minimal relation involving $\alpha\beta$, which is of the form
$$\alpha\beta + \alpha z \beta^* = 0,$$
for some $z\in J$, because $\alpha$ is the only arrow from $i$. We may replace $\beta$ by $\beta:= \beta + z\beta^*$ 
(adjusting presentation) and get $\alpha\beta=0$. Then in fact $\beta\La = \Omega^2(S_i)$ and hence 
$\Omega^3(S_i) = \Omega^{-1}(S_i) = \gamma\La$ which gives $\beta \gamma = 0$. \smallskip 

(1b) Further, note that $\gamma\alpha=\bar{\gamma}z\alpha^*$ for some non-zero $z\in\Lambda$. Actually, from 
$\gamma\alpha\prec I$ we get a minimal relation in $\Lambda$ of the form: 
$$\gamma\alpha = \bar{\gamma} z_1\gamma\alpha + \bar{\gamma}z_2\alpha^* + \gamma\alpha z_3\alpha^* \leqno{(*)}$$ 
We may assume $z_3=0$: if not we write as $\gamma\alpha u = \bar{\gamma}z_1\gamma\alpha + \bar{\gamma}z_2\alpha^*$, 
where $u = 1- z_3\alpha^*$ is a unit. We post multiply with $u^{-1} $ and get an expression analogous to (*) but 
with $z_3=0$.

We may also assume that $z_1=0$, otherwise we write $(1-\bar{\gamma}z_1)\gamma\alpha = \bar{\gamma}z_2\alpha^*$. 
We premultiply with the inverse of $1-\bar{\gamma}z_1$ and get
similarly an expression of the form
as stated in (1b) \medskip 

{\bf (2)} Now, let $f=e_x+e_j$ and consider the idempotent algebra $R=e\Lambda e$, where $e=f+e_i$. We claim that 
$R=KQ_R/I_R$ is given by the following quiver: 
$$
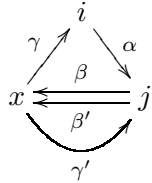
\xymatrix@R=0.6cm@C=0.45cm{&&i\ar[rd]^{\alpha}&& \\ 
&x\ar[ru]^{\gamma} \ar@/_20pt/[rr]_{\gamma'} &&j\ar@<-.5ex>[ll]_{\beta}\ar@<+.4ex>[ll]^{\beta'}&}$$
Clearly, $Q_R$ contains a triangle induced from the arrows $\gamma,\alpha,\beta$ in $Q$. In particular, we have also 
$\alpha\beta=\beta\gamma=0$ in $R$ and $\gamma\alpha\prec I_R$. The arrow $\beta$ cannot be in the right (or left) 
socle of $R$, and hence $Q_R$ admits at least one arrow $\gamma '\neq \gamma$ starting at $x$ (respectively, arrow 
$\alpha'\neq\alpha$ ending at $j$). Similarly, $\gamma$ and $\alpha$ are not in the socle, so we have additional 
arrows $\beta',\beta''\neq\beta$ starting at $j$ and ending at $x$, respectively. Note that $i$ is still a 
$1$-vertex in $Q_R$, and consequently, all the arrows $\alpha',\gamma',\beta',\beta''$ start and end at vertices 
$x$ or $j$. If $\beta'$ is not a loop, then we have double arrows from $j$ to $x$, and hence, due to tameness of 
$R$, we get that $\gamma',\alpha'$ are not loops. In this case, we must have one arrow $\gamma'=\alpha':x\to j$ and 
$\beta''=\beta'$, and $Q_R$ has required shape. \smallskip 

Therefore, it remains to exlude the case when $\beta'$ is a loop. If this is the case, then $\beta''$ is a loop, 
since otherwise, we get a wild subcategory formed by the double arrows $\beta,\beta'':j\to x$ and a loop $\beta'$. 
It follows that either $\gamma'=\alpha'$ is an arrow $x\to j$ or $\gamma'=\beta''$ and $\alpha'=\beta'$, and then 
$Q_R$ consists of triangle and two loops $\gamma'=\beta''$ at $x$ and $\beta'=\alpha'$ at $j$. In the first case, 
we obtain the following wild subcategory in covering: 
$$
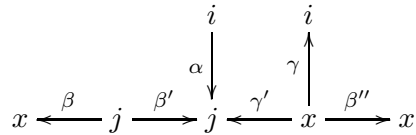
\xymatrix{ &&i\ar[d]_{\alpha}&i&\\
x&\ar[l]_{\beta} j \ar[r]^{\beta '}  &  j  & x \ar[u]^{\gamma}\ar[l]_{\gamma'}\ar[r]^{\beta''}& x}$$ 
In the second case, $\gamma\alpha\prec I_R$ implies that $\gamma\alpha= \gamma'z'\beta'$, where $z'$ must have a factor 
$\gamma\alpha$ since this is the only way to get from $x$ to $j$. So we actually have
$$\gamma\alpha = \gamma' z_1\gamma\alpha z_2\beta'$$
In general if $\theta = u\theta v$ with $u, v$ in the radical then it follows that $\theta = 0$, because 
$\theta = u\theta v = u(u\theta v)v = u^r\theta v^r$ and this eventually becomes zero. It follows that $\gamma\alpha=0$, 
but then $\gamma$ is in the socle, a contradiction. Therefore, $Q_R$ is of required form. \medskip 

{\bf (3)} Next, we prove that $\bar{i}=i^*$. Let $\bar{i}\neq i^*$, and consider the idempotent algebras 
$\bar{R}=\bar{e}\Lambda\bar{e}$, $R^*=e^*\Lambda e^*$, and $\hat{R}=\hat{e}\Lambda\hat{e}$, where $\bar{e}=f+e_{\bar{i}}$, 
$e^*=f+e_{i^*}$ and $\hat{e}=f+e_{\bar{i}}+e_{i^*}$. The quiver $Q_{\bar{R}}$ obviously has arrows $\beta:j\to x$ and 
$\bar{\gamma}: x\to \bar{i}$, hence a path from $j$ to $\bar{i}$. Moreover, $\bar{\gamma}$ is the unique arrow in 
$Q_{\bar{R}}$ starting at $x$, due to (1b). Because $\bar{R}$ is symmetric, we must have a path in $Q_{\bar{R}}$ from $x$ to $j$, 
hence an arrow $\delta: \bar{i}\to j$. \smallskip 

(3a) {\it We will show that the Gabriel quiver of $\bar{R}$ is of the form} 
$$
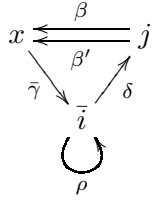
\xymatrix@R=0.6cm@C=0.45cm{
&x\ar[rd]_{\bar{\gamma}} &&j\ar@<-.5ex>[ll]_{\beta}\ar@<+.4ex>[ll]^{\beta'}& \\ 
&&\bar{i}\ar[ru]_{\delta}\ar@(dl, dr)[]_{\rho}&& \\ \\ }$$ 

First, observe that there is another arrow $\beta':j\to x$, $\beta'\neq\beta$, in $Q_{\bar{R}}$. Indeed, consider 
the idempotent algebra $\bar{S}=(\bar{e}+e_i)\Lambda(\bar{e}+e_i)$. Its quiver contains arrows 
$\gamma,\alpha,\beta,\delta$ and $\bar{\gamma}$. Moreover, both algebra $R$ and $\bar{R}$ are idempotent algebras of 
$\bar{S}$. We have the second arrow $\beta':j\to x$ in $Q_{R}$, so there must be a corresponding path in $Q_{\bar{S}}$, 
which does not factor through $\beta$. As a result, there is an arrow $\beta'\neq\beta$ in $Q_{\bar{S}}$ starting at 
$j$ and ending at $x$ or $\bar{i}$. In the first case, we get required double arrows in $Q_{\bar{S}}$, hence in 
$Q_{\bar{R}}$, so we have to exclude an arrow $\beta':j\to \bar{i}$ in $Q_{\bar{S}}$. In fact, then we have a path 
$\beta'\gamma'$ in $Q_{\bar{S}}$, $\gamma':\bar{i}\to x$, and the quiver $Q_{\bar{S}}$ contains the following subquiver. 
$$
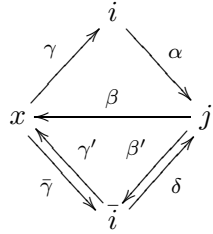
\xymatrix{ & i \ar[rd]^{\alpha}& \\ 
x \ar[ru]^{\gamma} \ar@<-0.4ex>[rd]_{\bar{\gamma}}&& j \ar@<-0.4ex>[ld]_{\beta'} \ar[ll]_{\beta} \\ 
&\bar{i} \ar@<-0.3ex>[ru]_{\delta} \ar@<-0.4ex>[lu]_{\gamma'} & }$$ 
We cannot have a loop $\mu$ at $\bar{i}$, because then we would have the following wild subcategory in covering: 
$$
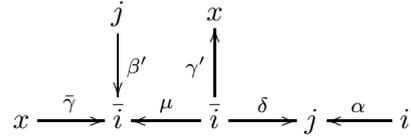
\xymatrix{& j\ar[d]^{\beta'} &x&& \\ 
x \ar[r]^{\bar{\gamma}} & \bar{i} & \ar[l]_{\mu} \bar{i} \ar[u]^{\gamma'} \ar[r]^{\delta} & j & \ar[l]_{\alpha} i }$$ 
Consequently, the above subquiver is the whole $Q_{\bar{S}}$ (no loops at $x,j$, since it would give loops in $Q_R$). 
But then the algebra $R_1=f\Lambda f=f\bar{S}f$ is given by the forbidden quiver 
$$\xymatrix@R=0.6cm@C=0.45cm{&x\ar@/_10pt/[rr] &&j\ar@<-.5ex>[ll]\ar@<+.4ex>[ll]& }$$ 
and we obtain a contradiction with Lemma \ref{lem:2vertices} (note that $\beta'\gamma'\nprec I_{\bar{S}}$, since 
otherwise, there is only one arrow $j\to x$ in $Q_R$). Therefore, we proved that $\beta'$ cannot be an 
arrow $j\to\bar{i}$, and hence, it must be the second arrow $j\to x$ in $Q_{\bar{S}}$, so also in $Q_{\bar{R}}$. \medskip 

In a similar way, one can prove that there is at least one further arrow in $Q_{\bar{R}}$ (besides $\beta,\beta',\bar{\gamma}$ 
and $\delta$). Otherwise consider the idempotent algebra $R_1= fRf$. This is tame symmetric given by the following quiver: 
$$
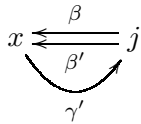
\xymatrix@R=0.6cm@C=0.45cm{
&x\ar@/_20pt/[rr]_{\gamma'} &&j\ar@<-.5ex>[ll]_{\beta}\ar@<+.4ex>[ll]^{\beta'}& }$$
and we obtain a contradiction with Lemma \ref{lem:2vertices} again. \smallskip 

As a result, we must have one more arrow in $Q_{\bar{R}}$. It cannot start at $x$, because all paths in $Q_\Lambda$ 
factor through $\bar{\gamma}$ (modulo $I$), so this is the unique arrow in $Q_{\bar{R}}$ starting at $x$. It cannot 
also start at $j$ (or end at $x$), otherwise this would yield a wild subcategory of $\bar{R}$. This leaves to consider the case 
when the new arrow starts at $\bar{i}$. \smallskip 

Suppose there is an arrow $\delta_1: \bar{i}\to j$. Since $\bar{R}$ is tame, one of $\bar{\gamma}\delta$, 
$\bar{\gamma}\delta_1$ is in some relation, and after adjusting $\delta_1$ (if necessary) we may assume 
$\bar{\gamma}\delta_1=0$ ($\bar{\gamma}$ is the unique arrow in $Q_{\bar{R}}$ starting from $\bar{i}$). 
Then $R_1 = f\bar{R}f$ is of the same form as above, and we get similar contradiction with Lemma \ref{lem:2vertices}. 
\smallskip

This actually proves that there is a loop $\rho$ at vertex $\bar{i}$ of $Q_{\bar{R}}$, and there are no other 
arrows, so $Q_{\bar{R}}$ indeed has the required shape. \medskip 

Using dual arguments, one can show that the idempotent algebra $R^*$ is given by the same quiver as $\bar{R}$:   
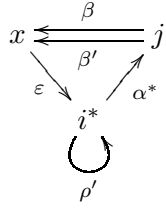
$$\xymatrix@R=0.6cm@C=0.45cm{
&x\ar[rd]_{\ve} &&j\ar@<-.5ex>[ll]_{\beta}\ar@<+.4ex>[ll]^{\beta'}& \\ 
&&i^*\ar[ru]_{\alpha^*}\ar@(dl, dr)[]_{\rho'}&& \\ \\ }$$ \medskip 

(3b) {\it We will prove that none of the paths $\bar{\gamma}\delta,\bar{\gamma}\rho$ and $\bar{\gamma}\rho\delta$ 
is involved in minimal relations defining $\bar{R}$.} \ First, suppose $\bar{\gamma}\rho\prec I_{\bar{R}}$. Because 
$\bar{\gamma}$ is the unique arrow in $Q_{\bar{R}}$ starting at $x$, we conclude that $\bar{\gamma}\rho=0$ (after 
possibly adjusting $\rho$). But then the idempotent algebra $R_1=f\bar{R}f$ is given by the quiver 
$$
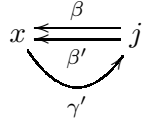
\xymatrix@R=0.6cm@C=0.45cm{
&x\ar@/_20pt/[rr]_{\gamma'} &&j\ar@<-.5ex>[ll]_{\beta}\ar@<+.4ex>[ll]^{\beta'}& }$$ 
which is impossible, due to Lemma \ref{lem:2vertices}. In the same way, one can show that 
$\bar{\gamma}\rho\delta\nprec I_{\bar{R}}$. \smallskip 

Now, let $\bar{\gamma}\delta\prec I_{\bar{R}}$. As above, after adjusting $\delta$, we can assume $\bar{\gamma}\delta=0$. 
We will show that $\rho^2\nprec I_{\bar{R}}$. Suppose to the contrary $\rho^2\prec I_{\bar{R}}$. We may assume that 
$\delta\beta'\nprec I_{\bar{R}}$, because we have double arrows. In this case, the local algebra $e_{\bar{i}}Re_{\bar{i}}$ 
is generated by $\rho$ and $Y=\delta\beta\bar{\gamma}$. Note that $\bar{\gamma}Y=0=Y\delta$, because $\bar{\gamma}\delta=0$, 
hence in particular, we have $Y^2=0$. Now, by our assumption, the radical of this local algebra is spanned by the 
monomials in $\rho,Y$ with $\rho$ and $Y$ alternating. So we can write 
$$\rho^2=c_0\rho Y+c_1\rho Y\rho+c_2\rho Y\rho Y+\dots +d_0Y+d_1Y_\rho+\dots = \rho U_1Y+\rho U_2\rho + YU_3Y + YU_4\rho,$$ 
where $U_2$ is in the radical. Then $\bar{\gamma}\rho^2\delta=\bar{\gamma}\rho U_2\rho\delta$ with each term in 
$J^2_{R_1}$, since $\bar{\gamma}\rho Y \rho \delta=\bar{\gamma}\rho\delta\cdot \beta \cdot \bar{\gamma}\rho\delta$. 
But then the path $\bar{\gamma}\rho^2\delta$ cannot induce an arrow in $Q_{R_1}$. Premultiplyind the above identity 
by $\rho$ and using the same argument, we conclude that $\bar{\gamma}\rho^3\delta$ is not an arrow in $Q_{R_1}$, 
and inductively, we get the same conclusion for any path of the form $\bar{\gamma}\rho^a\delta$. This shows that 
the quiver of $R_1$ has the form forbidden in Lemma \ref{lem:2vertices}. \smallskip 

This shows that $\rho^2$ is not involved in a minimal relation of $I_{\bar{R}}$. Therefore, in a covering of $\bar{R}$ 
we have the following subcategory 
$$\xymatrix{ &&&j\ar[d]_{\beta}&& \\ &&&x\ar[d]_{\bar{\gamma}}&& \\
j\ar[r]^{\beta} & x\ar[r]^{\bar{\gamma}} &\bar{i}&\ar[l]_{\rho}\bar{i}&\ar[l]_{\rho}\bar{i}\ar[r]^{\delta}& j}$$ 
All paths of lenght $2$ are non-zero, since the corresponding paths in $\bar{R}$ are not involved in relations (note 
that $\beta\bar{\gamma}\nprec I$, by Lemma \ref{lem:3.3}, since $i^*\neq\bar{i}$, and $\beta,\bar{\gamma}$ remain arrows 
in $Q_{\bar{R}}$, so also $\beta\bar{\gamma}\nprec I_{\bar{R}}$). Finally, the unique path of length $3$, i.e. 
$\beta\bar{\gamma}\rho$ is also not involved in a minimal relation of $I_{\bar{R}}$. If it was, we would get the 
following minimal relation in $\bar{R}$: 
$$\beta\bar{\gamma}\rho+\beta\bar{\gamma}\rho z=0,$$  
with $z\in J_{\bar{R}}$, because $\bar{\gamma}\delta=0$ and $\beta'\bar{\gamma}\prec I_{\bar{R}}$. Now adjusting 
$\rho:=\rho+\rho z$, we obtain $\beta\bar{\gamma}\rho=0$. Similarily, we have $\beta\bar{\gamma}\delta=0$, and 
consequently, the path $\beta\bar{\gamma}$ belongs to the (right) socle of $\bar{R}$, a contradiction. As a result, 
all paths in the above subquiver are not involved in minimal relations of $I_{\bar{R}}$, and hence, it gives a wild 
hereditary subcategory (in covering). \smallskip 

Summing up, it has been proven above that the paths $\bar{\gamma}\delta,\bar{\gamma}\rho$ and $\bar{\gamma}\rho\delta$ 
are not involved in minimal relations of $I_{\bar{R}}$. Using dual arguments for the second idempotent algebra $R^*$ 
one can show that the paths $\ve\alpha^*,\ve\rho'$ and $\ve\rho'\alpha^*$ are not involved in minimal relations of 
$I_{R^*}$. \medskip 

(3c) {\it We will find other wild subcategories, but this time in a covering of 
$\hat{R}$.}  Note first that the quiver $Q_{\hat{R}}$ of $\hat{R}$ contain arrows $\beta,\bar{\gamma},\alpha^*$ induced 
from arrows in $Q$, and additionally, the second arrow $\beta'$ appearing in all previous algebras $R,\bar{R},R^*$. 
The arrows $\beta,\beta',\bar{\gamma},\alpha^*$ exhaust all arrows in $Q_{\hat{R}}$ starting or ending at $x$ or $j$. 
It follows from symmetricity that there is at least one arrow $\xi:\bar{i}\to i^*$. We know from (3a) and (3a') 
that there are loops $\rho$ and $\eta$ in quivers of $\bar{R}$ and $R^*$, respectively, therefore we must have cyclic 
paths $\bar{i}\rightsquigarrow\bar{i}$ and $i^*\rightsquigarrow i^*$ in $Q_{\hat{R}}$ not containing $\beta$. In 
consequence, $Q_{\hat{R}}$ either: 1) admits an arrow $\xi':i^*\to \bar{i}$, or 2) two loops $\rho$ and $\eta$ at 
$\bar{i}$ and $i^*$. \smallskip 

1) In the first case, the quiver $\hat{Q}=Q_{\hat{R}}$ consists of arrows $\beta,\beta',\bar{\gamma},\alpha^*,\xi,\xi'$, 
and possibly, a second arrow $\zeta:\bar{i}\to i^*$. Indeed, all other arrows are arrows starting and ending at 
$\bar{i}$ or $i^*$, and we cannot have one more arrow $i^*\to \bar{i}$, because we would obtain a wild subquiver of 
type $K_2^-$. Hence, there is at most one additional arrow $\zeta:\bar{i}\to i^*$, and possibly loops at $\bar{i}$ 
or $i^*$. We can choose a basis $\cB$ of $\hat{R}$ consisting of paths in $\hat{Q}$ such that 
\begin{itemize}
\item the set $\bar{\cR}$ of all subpaths of paths in $\cB$ which start and end in $x,j,\bar{i}$ generates $\bar{R}$, and 
\item the set of paths $\bar{\cP}\subset\cR$ which are not products of shorter paths in $\bar{\cR}$ is a set of 
representatives of arrows in $Q_{\bar{R}}$. 
\end{itemize} 

Let now $\vf$ denote the path in $\bar{\cR}$ corresponding to the arrow $\delta$ in $Q_{\bar{R}}$. In case there 
is a loop $\rho:\bar{i}\to\bar{i}$ in $\hat{Q}$, it induces a loop $\rho$ in $Q_{\bar{R}}$, hence the path $\vf$ 
is always starting with $\xi$ or $\zeta$. It follows that there is no loop at $\bar{i}$. Indeed, if 
$\psi:\bar{i}\to\bar{i}$ is a loop in $\hat{Q}$, then we get the following wild subcategory in a covering of $\hat{R}$: 
$$\xymatrix{&&j\ar[d]_{\beta}&&&\\&&x\ar[d]_{\bar{\gamma}}&&& \\ 
j & i^*\ar[l]_{\alpha^*}\ar[r]^{\xi '}&\bar{i}&\bar{i}\ar[l]_{\psi}\ar@{-->}[r]^{\vf}&j&}$$ 
Note that the path $\vf$ is not involved in any relation defining $\hat{R}=K\hat{Q}/\hat{I}$, because otherwise, we 
have $\vf\prec I_{\bar{R}}$, which is impossible, since $\vf$ represents an arrow in $Q_{\bar{R}}$. Dually, using 
the path $\sigma$ in $\hat{Q}$, which corresponds to the arrow $\ve$ in $Q_{R^*}$, one can prove that there are no 
loops in $\hat{Q}$ at $i^*$. \smallskip 

As a result, we proved that $\hat{Q}$ consists of arrows $\beta,\beta',\bar{\gamma},\alpha^*,\xi,\xi'$ and possibly 
$\zeta:\bar{i}\to i^*$. If we have double arrows $\xi,\zeta:\bar{i}\to i^*$, then one of the paths $\xi\alpha^*$ 
or $\zeta\alpha^*$ is involved in a minimal relation of $\hat{I}$, since otherwise, we would get a wild subcategory 
of type $K_2*$. Without loss of generality, assume that $\zeta\alpha^*\prec\hat{I}$, but $\xi\alpha^*\nprec I$. 
Then any path in $\cB$ from $\bar{i}$ to $j$ is of the form $(\xi\xi')^k\alpha^*$, hence independently of $\cB$, 
we will always have $\xi\xi'$ and $\xi\alpha^*$ in $\bar{\cP}$. Similarily, one can take a subset $\cP^*\subset \cB$ 
which gives the arrows of $Q_{R^*}$, and then $\xi'\xi$ and $\bar{\gamma\xi}$ will be always in $\cP^*$. Consequently, 
we may identify the induced arrows $\delta,\rho$ in $Q_{\bar{R}}$ (respectively, arrows $\ve,\rho'$ in $Q_{R^*}$) 
with the following paths in $\hat{Q}$: 
$$\delta=\xi\alpha^*,\quad \rho=\xi\xi',\quad \ve=\bar{\gamma}\xi, \quad\mbox{and}\quad \rho'=\xi'\xi.$$ 
Therefore, using (3b) we conclude that there is a wild (hereditary) subcategory in covering of $\hat{R}$ of the 
following the form: 
$$\xymatrix{&&j&x\ar[d]_{\bar{\gamma}}& \\ 
x\ar[r]^{\bar{\gamma}} &\bar{i}\ar[r]^{\xi}&i^*\ar[r]^{\xi'}\ar[u]^{\alpha^*}&\bar{i}\ar[r]^{\xi}&i^*}$$ \medskip 

2) For the second case, when $\hat{Q}$ admits two loops $\rho,\rho'$, we clearly have no more arrows in $\hat{Q}$, 
since adding any other arrow makes the algebra wild. In this case, we may fix a basis of $\hat{R}$ containing 
paths $\bar{\gamma}\xi$ and $\xi\alpha^*$, which after passing to idempotent algebras induce arrows $\epsilon$ in 
$Q_{R^*}$ and $\delta$ in $Q_{\bar{R}}$. Further, this basis cannot contain longer paths of the form $\bar{\gamma}\rho^p\xi$ 
or $\xi(\rho')^p\alpha^*$, for $p\geqslant 1$, because then we would get more than one arrow $x\to i^*$ in $Q_{R^*}$ 
or more than one arrow $\bar{i}\to j$ in $Q_{\bar{R}}$. As a result, the algebra $R_1=f\hat{R}f$ is given by the following 
quiver 
$$\xymatrix@R=0.6cm@C=0.45cm{&x\ar@/_10pt/[rr] &&j\ar@<-.5ex>[ll]\ar@<+.4ex>[ll]& }$$ 
where the unique arrow $x\to j$ is identified with the path $\bar{\gamma}\xi\alpha^*$ in $\hat{Q}$. This contradicts 
Lemma \ref{lem:2vertices}. \medskip 

Consequently, we have proven that indeed $\bar{i}=i^*$. In the remaining two steps we will show that this vertex is 
$2$-regular (in $Q$) and that $a=b$ is $1$-regular. \bigskip 

{\bf (4)} We show that $\bar{i}$ is a $2$-vertex in $Q$:  Indeed, if $\bar{i}$ was a $1$-regular vertex, we 
would get $\gamma\alpha$ involved in a minimal relation of the form 
$$\gamma\alpha=\bar{\gamma}\alpha^*\zeta,$$  
for some $\zeta$ is in $e_j\Lambda e_j$. Now, consider the exact sequence for $S_{\bar{i}}$:  
$$0\to S_{\bar{i}}\to 
\xymatrix{P_{\bar{i}} \ar[r]^{\bar{\gamma}} & P_x \ar[r]^{d_2} & P_j \ar[r]^{\alpha^*} & P_{\bar{i}}}\to S_{\bar{i}} \to 0.$$ 
The middle map $d_2$ is given by an element $\bar{m}\in e_j J e_x$ satisfying $\bar{m}\bar{\gamma}=0$. But then, 
we obtain $\bar{m}\gamma\alpha=\bar{m}\bar{\gamma}\alpha^*\zeta=0$. Because $\alpha$ is the unique arrow starting at 
$i$, we conclude that $\bar{m}\gamma$ is a socle element in $e_j\Lambda e_i$, a contradiction. Therefore, $i^*=\bar{i}$ 
is indeed a $2$-regular vertex in $Q$. 

\bigskip 

{\bf (5)} We will show that $a=b$. Assume that this is not the case, and take the 
idempotent algebra $\wt{R}=\wt{e}\Lambda\wt{e}$, given by $\wt{e}=e+e_{\bar{i}}+e_a+e_b$. \smallskip 

We will prove first that its Gabriel quiver $\wt{Q}=Q_{\wt{R}}$ admits a loop at vertex $\bar{i}$. Indeed, 
as in the proof of Lemma \ref{lem:4.1}, we may fix a basis $\cB$ of $\Lambda$ such that the set $\cR$ of 
subpaths of elements in $\cB$ with source and target in $x,j,i,\bar{i},a,b$ generates $\wt{R}$ and the 
associated subset $\cP\subset\cR$ of paths, which are not products of shorter paths in $\cR$, represent 
the set of arrows of $\wt{Q}$. \smallskip 

Now, consider the same construction for an idempotent algebra $\bar{R}=\bar{e}\tilde{R}\bar{e}=\bar{e}\La\bar{e}$. 
As above, we may pick a basis $\wt{\cB}$ of $\wt{R}$ (consisting of paths in $\wt{Q}$) such that the set $\wt{\cR}$ 
of subpaths of elements in $\wt{\cB}$ with source and target in $x,j,\bar{i}$ generates $\bar{R}$ and the 
associated subset $\wt{\cP}\subset\wt{\cR}$ of paths, which are not products of shorter paths in $\wt{\cR}$, 
determine the arrows of $Q_{\bar{R}}$. \smallskip 

Recall that the Gabriel quiver $Q_{\bar{R}}$ is known by (3a), we showed that it consists of arrows 
$\beta,\beta',\bar{\gamma},\alpha^*$ and a loop $\rho$. The first three arrows are arrows in both $Q$ and $\wt{Q}$, 
and the remaining two correspond to paths in $\wt{\cP}$. 

The arrow $\beta':j\to x$ is induced from a path in $\wt{Q}$ from $j$ to $x$, which does not pass through vertices 
$x,j,\bar{i}$. It cannot pass through $i$, since it would then must have passed through $x$ earlier, so it can only 
pass through $a$ and $b$. It follows, that there is an arrow $\zeta:b\to a$ in $\wt{Q}$ and $\beta'$ is represented 
by a path $\bar{\beta}u\beta^*$ with $u$ a path of the form 
$$\xymatrix{b\ar[r]^{\zeta} & a \ar[r] & b \ar[r] & \dots \ar[r] & a},$$ 
where we do not indicate possible loops at $a$ or $b$, which may appear in $u$. 

Further, we have a loop $\rho:\bar{i}\to\bar{i}$ in $\bar{R}$, which corresponds to a path from $\bar{i}$ to 
$\bar{i}$ in $\wt{Q}$, passing only through vertices $a,b$. Consequently, $\rho$ is represented by 
one of the following paths in $\wt{Q}$: 
\begin{enumerate}
\item[a)] $\rho=\sigma v\eta$, for some arrows $\sigma:\bar{i}\to b,\eta:a\to\bar{i}$ and a path 
$v:b\to a \to b\to \dots \to a$, 
\item[b)] $\rho=\sigma w\eta$, for some arrows $\sigma:\bar{i}\to a,\eta:b\to\bar{i}$ and a path 
$w:a\to b \to a \to \dots \to b$, or 
\item[c)] $\rho$ is a loop in $\wt{Q}$.  
\end{enumerate} 

We will exclude a) and b)

{\it Case a)} \ Then 
the subquiver of $\wt{Q}$ formed by arrows $\alpha,\alpha^*,\bar{\beta}$ and $\sigma$ gives rise 
to a wild subcategory if $\alpha^*\bar{\beta}\nprec I_{\wt{R}}$, so we have $\alpha^*\bar{\beta}\prec I_{\wt{R}}$. 
Dually, using subquiver formed by $\gamma,\bar{\gamma},\beta^*$ and $\eta$, we deduce that 
$\beta^*\bar{\gamma}\prec I_{\wt{R}}$. Because $\alpha^*,\beta^*,\bar{\beta}$ and $\bar{\gamma}$ are arrows 
of $Q$, so these paths are also involved in minimal relations of $I$. Consequently, due to Lemma \ref{lem:3.3}, 
we obtain two more arrows $\sigma^*:b\to\bar{i}$ and $\eta^*:\bar{i}\to a$ in $Q$, and hence, also in $\wt{Q}$. 
But then, we get the following wild subcategory of type $\wt{\wt{\bE}}_6$ (in a covering of $\wt{R}$): 
$$\xymatrix{ && i &&& \\ && x \ar[d]_{\bar{\gamma}} \ar[u]^{\gamma} &&& \\ 
x & \ar[l]_{\beta^*} a \ar[r]^{\eta} & \bar{i} & \ar[l]_{\sigma^*} b \ar[r]^{\zeta} & a & \ar[l]_{\eta^*} \bar{i} }$$ \smallskip 

{\it Case  b)} \ Then $\wt{Q}$ admits at least one arrow $\zeta':a\to b$. It follows from relations 
in (1b) that $\gamma\alpha=\bar{\gamma}z\alpha^*$, so we have $e_jJ_{\wt{R}}^2=\bar{\gamma}J_{\wt{R}}$. Note 
that we have no more arrows in $\wt{Q}$ between $a$ and $b$, because otherwise, we get a wild subquiver of type 
$K_2^\pm$. Moreover, there are no loops at $a$ and $b$, since a loop at $a$ or $b$ easily gives a wild hereditary 
subcategory of type $\wt{\wt{\bE}}_6$ with three arrows meeting at $a$ or $b$. 

Hence $\wt{Q}$ consists of the seven arrows $\gamma,\alpha,\beta,\bar{\gamma},\alpha^*,\bar{\beta},\beta^*$ 
from $Q$ and $\sigma,\eta,\zeta,\zeta'$. Observe also that the path $\rho$ in $\wt{Q}$ corresponds to a loop  
$\rho$ in $Q_{\bar{R}}$, and it has the form $\rho=\sigma\zeta'\zeta\dots \zeta'\eta$, hence we infer that both 
$\sigma\zeta'\nprec I_{\wt{R}}$ and $\zeta'\eta\nprec I_{\wt{R}}$. It follows that $\beta^*\bar{\gamma}\prec I_{\wt{R}}$, 
since otherwise, we obtain the following wild subcategory 
$$\xymatrix@R=0.3cm@C=0.4cm{a \ar[rr]^{\zeta'} \ar[dd]_{\beta^*} && b \ar@/^20pt/[lddd]_{\eta}\\ &i& \\ 
x\ar[rd]_{\bar{\gamma}} \ar[ru]^{\gamma} && \\ & \bar{i} & }$$ 

Similarily, using $\sigma\zeta'\nprec I_{\wt{R}}$, we deduce that $\alpha^*\bar{\beta}\prec I_{\wt{R}}$. But 
$\alpha^*,\beta^*,\bar{\gamma}$ and $\bar{\beta}$ are arrows in $Q$, so in fact, both paths $\beta^*\bar{\gamma}$ 
and $\alpha^*\bar{\beta}$ are involved in minimal relations of $I$. Consequently, it follows from Lemma \ref{lem:3.3} 
that $\sigma$ and $\eta$ are actually arrows in $Q$, hence we conclude that also $\bar{\gamma}\sigma\prec I$, due 
to the Triangle Lemma. Now, because any path from $x$ to $a$ factors through $\bar{\gamma}$, we obtain a relation in $I_{\wt{R}}$ 
of the form $$\bar{\gamma}\sigma+\bar{\gamma}z=0$$ 
with $z\in J^{2}_{\wt{R}}$. But then, substituting $\sigma:=\sigma+z$, we get a new presentation with $\bar{\gamma}\sigma=0$. 

Now, we can similarily show that $\bar{\gamma}\alpha^*\prec I_{\wt{R}}$ implies $\bar{\gamma}\alpha^*=0$ (possibly 
adjusting $\alpha^*$), and then we get an arrow $\bar{\gamma}$ in the socle of $\wt{R}$. Therefore, we must have 
$\bar{\gamma}\alpha^*\nprec I_{\wt{R}}$. 

As a result, we conclude that $\bar{\gamma}\alpha^*$ is the unique path in $\cB$ which does not factor 
through $x,j$. But then, the associated idempotent algebra $R_1=f\wt{R}f$ is given by the following quiver 
$$\xymatrix@R=0.6cm@C=0.45cm{&x\ar@/_10pt/[rr] &&j\ar@<-.5ex>[ll]\ar@<+.4ex>[ll]& }$$ 
where the unique arrow $x\to j$ represents the path $\bar{\gamma}\alpha^*$, so we obtain a 
contradiction with Lemma \ref{lem:2vertices}. This finally proves that cases a) and b) cannot happen, and 
we are left with c). This completes the proof of (5).

\medskip

It follows that $\wt{Q}$ has a subquiver of the form: 
$$
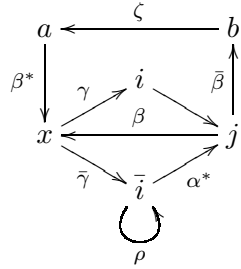
\xymatrix@R=0.2cm{a \ar[dd]_{\beta^*} & & b\ar[ll]_{\zeta} \\ & i\ar[rd] & \\ 
x \ar[ru]^{\gamma} \ar[rd]_{\bar{\gamma}}  && j \ar[ll]_{\beta} \ar[uu]^{\bar{\beta}} \\ 
&\bar{i} \ar[ru]_{\alpha^*} \ar@(dl,dr)[]_{\rho} & \\ && }$$ 
where $\gamma,\alpha,\beta,\bar{\gamma},\bar{\beta},\alpha^*,\beta^*$ are all arrows in $\wt{Q}$ that start or end 
at vertices $x,i,j$. \medskip 

{\bf (6)} We will investigate arrows starting or ending with $a$ or $b$. \smallskip 

Observe first that there is no arrow $\tau:a\to\bar{i}$ in $\wt{Q}$, because otherwise, we would get a wild subcategory 
of the form 
$$
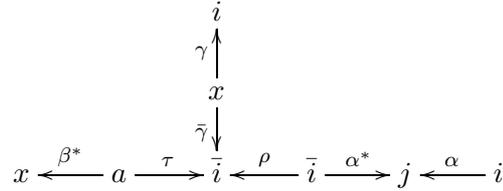
\xymatrix@R=0.6cm{&&i&&& \\ && x \ar[d]_{\bar{\gamma}} \ar[u]^{\gamma} &&& \\ 
x & a \ar[l]_{\beta^*} \ar[r]^{\tau} & \bar{i} & \bar{i} \ar[l]_{\rho} \ar[r]^{\alpha^*} & j & i \ar[l]_{\alpha}}$$ 
Using dual arguments, one can similarily prove that there is no arrow $\bar{i}\to b$ in $\wt{Q}$. \medskip 

{\bf (7)} We will prove that  $a$ and  $b$ cannot both be  $1$-regular  in $\wt{Q}$. Assume for a contradiction that  
both $a,b$ are $1$-regular vertices of $\wt{Q}$. First, using exact sequences for periodic simple modules 
$S_x,S_j$ and $S_i$ in $\mod\La$, we obtain the following equalities of dimension vectors 
$$p_b+p_x=p_j^+=p_j^-=p_i+p_{\bar{i}}=p_x^+=p_x^-=p_a+p_j \ \mbox{ and } \ p_x=p_j.$$ 
As a result, we get $p_a=p_b$, and hence, also dimension vectors of corresponding projective 
$\wt{R}$-modules are equal: $[e_a\wt{R}]=[e_b\wt{R}]$. Now, because $\zeta$ is the uniuqe arrow 
in $\wt{Q}$ starting at $b$, we have $\zeta\wt{R}\simeq\Omega_{\wt{R}}(S_b)$, so 
$[\zeta\wt{R}]=[e_b\wt{R}]-s_b$. On the other hand, $\zeta$ is the unique arrow in $\wt{Q}$ ending at 
$a$, so $\zeta\wt{R}=\Omega_{\wt{R}}^{-1}(S_a)$ and it has dimension vector $[\zeta\wt{R}]=[e_a\wt{R}]-s_a$. 
Involving the previous euality, we get $s_a=s_b$, a contradiction. This proves
the claim. \medskip 

{\bf (8)} We will show now that this leads to a contradiction in case $a$ is not a $1$-vertex. The arguments in case when 
$b$ is not $1$-regular are dual. \smallskip 

Suppose first that $|a^-|\geqslant 2$, and let $\eta:y\to a$ be the second arrow $\eta\neq\zeta$ ending at $a$. 
If $y=\bar{i}$, then we obtain the following wild subcategory in the covering of $\wt{R}$. 
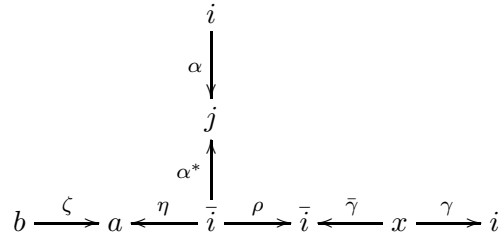
$$\xymatrix{&&i\ar[d]_{\alpha}&&& \\ &&j&&& \\ 
b\ar[r]^{\zeta}& a & \bar{i} \ar[l]_{\eta} \ar[u]^{\alpha^*} \ar[r]^{\rho} & \bar{i}  & 
\ar[l]_{\bar{\gamma}} x \ar[r]^{\gamma} & i}$$
Hence, we may assume that $\eta$ starts at $b$ or $a$. In particular, there are no arrows $\bar{i}\to a$, so we 
conclude as above that $\beta^*\bar{\gamma}\nprec\wt{I}$. Dually, we can prove that there are no arrows 
$b\to\bar{i}$, so $\alpha^*\bar{\beta}\nprec \wt{I}$. \smallskip 

If $\eta$ is a loop $\eta:a\to a$, then we get the following wild subcategory 
$$\xymatrix{ &&i&&&i\ar[d]_{\alpha}& \\ 
a & a \ar[l]_{\eta} \ar[r]^{\beta^*} & x\ar[u]^{\gamma} \ar[r]^{\bar{\gamma}} & \bar{ i} & 
\bar{i} \ar[l]_{\rho} \ar[r]^{\alpha^*} & j\ar[r]^{\bar{\beta}}& b}$$ 
Therefore, we can assume that $\eta$ is an arrow $b\to a$, and we have double arrows from $\zeta,\eta:b\to a$. 
Moreover, if there is an arrow $\sigma:a\to z$ different from $\beta^*$, then replacing $\eta$ by $\sigma$, we 
get a wild subcategory of the same type. Hence, one may assume that $\beta^*$ is the unique arrow in $\wt{Q}$ 
starting at $b$. In this case, one of $\zeta\beta^*$ or $\eta\beta^*$ is involved in a minimal relation of $\wt{I}$, 
since otherwise, we get a wild subquiver of type $K_2^*$. Since $\beta^*$ is the unique arrow in $\wt{Q}$ starting 
at $a$, we conclude that exactly one, say $\zeta\beta^*\prec\wt{I}$ and $\eta\beta^*\nprec \wt{I}$. Then we have 
also $\eta\beta^*\bar{\gamma},\eta\beta^*\gamma\nprec \wt{I}$. Indeed, if 
$\eta\beta^*\bar{\gamma}\prec I$, then it is involved in a minimal relation of the form 
$$\eta\beta^*\bar{\gamma}=\eta\beta^*\bar{\gamma}\Sigma$$ 
with $\Sigma\in J_{\wt{R}}$. Hence, by repeated substitution, we get $\eta\beta^*\bar{\gamma}=0$. It follows that 
either $\eta\beta^*\gamma=0$, and then $\eta\beta^*$ is in the socle, or $\eta\beta^*\gamma\neq 0$ is in the socle 
of $\wt{R}$. Now, suppose that $\eta\beta^*\gamma\prec I$. Then adjusting $\gamma$, we get $\eta\beta^*\gamma=0$, 
and in this case $\beta^*\gamma\neq 0$ is in the socle of $\wt{R}$ (we have $\zeta\beta^*\gamma=0$, because 
$\zeta\beta^*$ is involved in a minimal relation of the form $\zeta\beta^*=z_1\eta\beta^*+z_2\beta$). \smallskip 

As a result, we obtain the following wild subcategory in a  covering: 
$$\xymatrix{ && i &&& i\ar[d]_{\alpha} & \\ 
b \ar[r]^{\eta} & a \ar[r]^{\beta^*} & x \ar[r]^{\bar{\gamma}}\ar[u]^{\gamma} & \bar{i} & 
\ar[l]_{\rho} \bar{i} \ar[r]^{\alpha^*} & j \ar[r]^{\bar{\beta}} & b }$$ 
\medskip 

Finally, we are left with the case when $a$ is not $1$-regular in $\wt{Q}$, but $|a^-|=1$. Then $\zeta$ is the 
unique arrow in $\wt{Q}$ ending at $a$ and there is an arrow $a\to y$ different from $\beta^*$. In particular, 
it follows that $\beta^*\bar{\gamma}\nprec\wt{I}$ (no arrows $\bar{i}\to a$), and we obtain the same wild subcategory 
as in the case of a loop $\eta:a\to a$, but we use the second arrow $a\to y$ (different from $\beta^*$), instead of 
the resolved loop $\eta$. \medskip  

This eventually shows that $a=b$. \bigskip

{\bf (9)} Finally, it remains to prove that $a=b$ is a $1$-vertex in $Q$. Suppose that $a$ is a $2$-vertex with 
$a^+=\{\beta^*,\eta\}$ and $a^-=\{\bar{\beta},\sigma\}$. Let first both $\alpha^*\bar{\beta}\prec I$ and 
$\beta^*\bar{\gamma}\prec I$. It follows from Lemma \ref{lem:3.3} that then we have two arrows $\eta:a\to \bar{i}$ 
and $\sigma:\bar{i} \to a $, and hence, the quiver $Q$ has a form forbidden in Corollary \ref{exept.5}. \medskip   

Consequently, one of $\alpha^*\bar{\beta}$ or $\beta^*\bar{\gamma}$ is not involved in a minimal relation. In fact, 
we can assume that both $\alpha^*\bar{\beta},\beta^*\bar{\gamma}\nprec I$, because if $\alpha^*\bar{\beta}\nprec I$ 
but and $\beta^*\bar{\gamma}\prec I$, then by Lemma \ref{lem:3.3}, we conclude that there is an arrow $\sigma:\bar{i}\to a $, 
and hence, the following wild subactegory in covering. 
$$
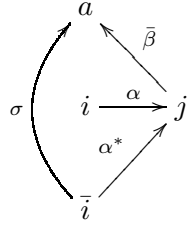
\xymatrix{ & a & \\ & i \ar[r]^{\alpha} & j \ar[lu]_{\bar{\beta}} \\ & \ar@/^20pt/[uu]^{\sigma} \bar{i} \ar[ru]^{\alpha^*} &}$$ 
\smallskip

Thus, we have both $\alpha^*\bar{\beta}\nprec I$ and $\beta^*\bar{\gamma}\nprec I$. Applying Lemma \ref{lem:3.4}, 
we deduce that also $\sigma\beta^*\nprec I$ and $\bar{\beta}\eta\nprec I$, because $i$ is a $1$-vertex. On the 
other hand, there are no arrows $x\to j$ in $Q$, so using Lemma \ref{lem:3.3} again, we get $\bar{\beta}\beta^*\nprec I$, 
and therefore, there is at most one path of length $2$ passing through $a$, which is involved in a minimal relation 
of $I$ (and it is $\sigma\eta$). \smallskip 

As a result, we obtain the following wild one-relation algebra in covering (see RiII): 
$$
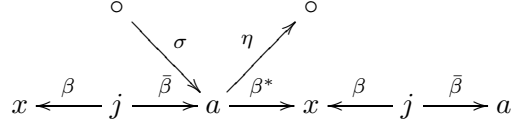
\xymatrix{ & \circ \ar[rd]^{\sigma} & & \circ && \\ 
x & \ar[l]_{\beta} j \ar[r]^{\bar{\beta}} & a \ar[r]^{\beta^*} \ar[ru]^{\eta} & x & \ar[l]_{\beta} j \ar[r]^{\bar{\beta}} & a } $$ 

\end{proof} \bigskip 

\section{Proof of the main theorem}\label{sec:5}

{\bf (1)} We start by setting the scene.
Assume $i\in Q_0$ is a $1$-vertex, i.e. $Q$ has a subquiver 
$$\xymatrix{x\ar[r]^{\gamma}&i\ar[r]^{\alpha} & j}$$ 

Obviously, if $x=j$ then $i$ belongs to a block $B$ of type V1 and $Q$ is glueing of $B$ with the rest part of the quiver, so the 
claim holds in this case. We assume further that $x\neq j$ and we will show that then $i$ belongs to a block of type V2. \medskip 

It is clear from Lemma \ref{lem:3.5} that then $x$ and $j$ are $2$-regular vertices of $Q$, and hence $Q$ has the following 
subquiver 
$$\xymatrix{
\bar{i} &\ar[l]_{\bar{\gamma}} x \ar[r]^{\gamma} & i \ar[r]^{\alpha} & j \ar[d]_{\nu}\ar[rd]^{\bar{\nu}} & i^*\ar[l]_{\alpha^*} \\ 
a\ar[ru]^{\delta}&b\ar[u]_{\delta^*}&&c&d}\leqno(*)$$ \smallskip 

If $\gamma\alpha\prec I$, then $i$ belongs to a triangle, by Lemma \ref{lem:3.3}, hence it follows from Proposition 
\ref{prop:4.2} that $i$ belongs to a block of required shape (the quiver in \ref{prop:4.2} is a glueing of a block of 
type V2 containing $i$ with a triangle). Therefore, we may assume that $i$ is not a part of a triangle in $Q$. In other words, 
vertices $a$ and $b$ are different from $j$, or equivalently, vertices $c,d$ are different from $x$; in particular, by Lemma 
\ref{lem:3.3}, we get $\gamma\alpha\nprec I$ and both $\alpha\nu\nprec I$ and $\alpha\bar{\nu}\nprec I$. Similarly, 
$\delta\gamma\nprec I$ and $\delta^*\gamma\nprec I$. Moreover, we have $a\neq b$ and $c\neq d$, since otherwise, we 
would get a wild subquiver of type $K_2^*$. \smallskip 

{\bf (2)} We will show that $\{a,b\}\cap\{c,d\}$ is not empty.
Suppose to the contrary that $\{a,b\}\cap\{c,d\}$ is empty. Then 
it follows from Lemma \ref{lem:3.6} that for every path $\pi$ of length $3$ in $(*)$, we have $\pi\nprec I$. Moreover, 
applying \cite[Lemma 4.7]{EHS}, we conclude that also all paths of length $4$ in $(*)$ are not involved in minimal relations of $I$, 
since they all pass through a $1$-vertex $i$. Therefore, we have a tame hereditary subcategory $\mathcal{C}$ of type 
$\widetilde{\bD}_6$, given by vertices $x,i,j,a,b,c,d$ and arrows $\alpha,\gamma,\delta,\delta^*,\nu$ and $\bar{\nu}$. \smallskip

Now, it follows that both $\bar{i}$ and $i^*$ are $2$-regular. To prove this, suppose one of them, say $\bar{i}$, is a $1$-regular 
vertex. We claim that then both $\delta\bar{\gamma}\nprec I$ and $\delta^*\bar{\gamma}\nprec I$. Indeed, if one of the paths is 
involved in a minimal relation, then by Lemma \ref{lem:3.3}, we obtain a triangle 
$\xymatrix@C=0.4cm{\bullet \ar[r]& x\ar[r]^{\bar{\gamma}} & \bar{i}\ar[r] & \bullet}$, where $\bullet=a$ or $b$, so $a$ or $b$ 
is a $2$-vertex, due to Lemma \ref{lem:3.5} (a triangle with $1$-vertex $\bar{i}$). But then there is an arrow $a\to a'$ 
different from $\delta$ or an arrow $b\to b'$ different from $\delta^*$, which gives a wild hereditary subcategory of 
type $\wt{\wt{\bD}}_6$ (obtained from $\mathcal{C}$ by adding the new arrow together with a vertex $a'$ or $b'$). This 
shows that both $\delta\bar{\gamma}\nprec I$ and $\delta^*\bar{\gamma}\nprec I$, which yields another contradiction, since then 
we get a wild hereditary subcategory of type $\wt{\wt{\bD}}_4$, given as follows 
$$\xymatrix@R=0.4cm{&\bar{i}&& \\ 
a\ar[r]^{\delta}&x\ar[r]^{\gamma}\ar[u]^{\bar{\gamma}}&i\ar[r]^{\alpha}& j\\ &b\ar[u]^{\delta^*}&& }$$ 
As a result, it has been proved that $\bar{i}$ is $2$-regular, and hence we have an arrow $i'\to\bar{i}$ in $Q$. Using dual 
arguments, one can prove that $i^*$ is $2$-regular, so there is an arrow $i^*\to i''$. One can also assume that all $a,b,c,d$ 
are $1$-regular, since otherwise, as above $\mathcal{C}$ can be extended to a wild hereditary subcategory of type 
$\wt{\wt{\bD}}_6$. So $a,b,c,d$ are $1$-vertices. \smallskip 

Moreover, using an  analogous argument with $\wt{\wt{\bD}}_4$, we deduce that $\delta\bar{\gamma}\prec I$ or 
$\delta^*\bar{\gamma}\prec I$ and $\alpha^*\nu\prec I$ or $\alpha^*\bar{\nu}\prec I$. Actually, we may assume that all of the 
paths $\delta\bar{\gamma}$, $\delta^*\bar{\gamma}$, $\alpha^*\nu$ and $\alpha^*\bar{\nu}$ are involved in minimal relations. 
Indeed, if one of them, say $\delta\bar{\gamma}\nprec I$, then $\delta^*\bar{\gamma}\prec I$ and we obtain the following 
subcategory 
$$\xymatrix@R=0.3cm{i'\ar[r]&\bar{i} & a\ar[d]^{\delta}&&&& \\ 
&&x\ar[lu]_{\bar{\gamma}}\ar[r]_{\gamma}&i\ar[r]_{\alpha}&j&i^*\ar[l]^{\alpha^*}\ar[r]&i'' \\ 
&b\ar[ru]_{\delta^*}& &&&& }$$
being a wild one-relation algebra RiVI. \smallskip 

Concluding, $Q$ has a subquiver of the following form 
$$\xymatrix@R=0.7cm{&a\ar[rd]^{\delta}&&&&c\ar[d]^{}&\\ 
i'\ar[r]&\bar{i}\ar[u]^{\sigma}\ar[d]_{\bar{\sigma}}&x\ar[l]^{\bar{\gamma}} 
\ar[r]_{\gamma}&i\ar[r]_{\alpha}&j\ar[ru]^{\nu}\ar[rd]_{\bar{\nu}}&i^*\ar[l]^{\alpha^*}\ar[r]&i'' \\ 
&b\ar[ru]_{\delta^*}&&&&d\ar[u]&}$$ 
Now, applying Proposition \ref{prop:4.2} to a $1$-vertex $a$ in a triangle, we see that one of the successors of $\bar{i}$ 
must be a $2$-vertex, which gives a contradiction, since here both $a,b$ are $1$-regular. Therefore, it has been proven 
that $\{a,b\}$ and $\{c,d\}$ share at least one vertex. Without loss of generality, we assume that $a=c$. \bigskip 

{\bf (3) } The next step is to show  that then $b\neq d$: 
 Assume that $b=d$. We will show that then one of $a$ or $b$ must be a $1$-regular vertex, and then it would 
follow that $i$ belongs to a block of type V2, as required. Suppose to the contrary that both $a$ and $b$ are 
$2$-vertices. Then there are arrows $\eta:y\to a, \eta':a\to y'$ and $\sigma:z\to b,\sigma':b\to z'$ 
with $\eta\neq\nu$, $\eta'\neq\delta$, $\sigma\neq\bar{\nu}$ and $\sigma '\neq \delta ^*$. Now it is easy to see 
that one of the paths $\eta\delta$, $\nu\eta'$, $\sigma\delta^*$ or $\bar{\nu}\sigma'$ is not involved in a minimal 
relation. Indeed, if all the paths are involved in minimal relations, then it follows from Lemma \ref{lem:3.3} that 
$y=z=\bar{i}$ ($i=y$ or $z$ implies $j=a$ or $b$) and similarly $y'=z'=i^*$, and in this case $\Lambda$ admits the 
following wild subcategory 
$$\xymatrix{i&\ar[l]^{\gamma} x & \ar[l]^{\delta}\ar[d]^{\eta'} a \\ 
& b\ar[u]^{\delta^*} \ar[r]^{\sigma'}& i^* }$$
Hence one of the paths, say $\eta\delta$, is not involved in a minimal relation of $I$. But then $\Lambda$ admits 
another subcategory of the same shape: 
$$\xymatrix{y\ar[r]^{\eta}& a\ar[d]_{\delta}  & \ar[l]_{\nu}\ar[d]^{\bar{\nu}} j \\ 
& x &\ar[l]^{\delta^*} b }$$ 
but not necessarily wild. Clearly, $\Lambda$ is wild if both $\nu\delta\nprec I$ and $\bar{\nu}\delta^*\nprec I$. 
We shall show that this is the case. In fact, if $\nu\delta\prec I$, then by Lemma \ref{lem:3.3}, there is an arrow 
$\bar{\gamma}=\alpha^*:x\to j$, and hence $j=\bar{i}$ and $x=i^*$. Further, we conclude from Lemmas \ref{lem:3.3} 
and \ref{lem:3.6} that $\bar{\nu}\sigma'\nprec I$ and $\alpha\bar{\nu}\sigma '\nprec I$, because $a\neq b$ and 
$i$ is a $1$-vertex not lying in a triangle. Using dual arguments, we see that also $\sigma\delta^*\nprec I$ and 
$\sigma\delta^*\gamma\nprec I$. In this case, we must have $\sigma\sigma'\prec I$, because otherwise we get the 
following wild subcategory of type $\wt{\wt{\bD}}$: 
$$\xymatrix{&&i\ar[d]_{\alpha}&z\ar[d]_{\sigma}& \\ 
y\ar[r]^{\eta}&a&j\ar[l]_{\nu}\ar[r]^{\bar{\nu}}&b\ar[r]^{\sigma'}& z'}$$ 
As a result, vertices $z$ and $z'$ lie in one triangle, and hence, both cannot be $1$-regular, by Lemma \ref{lem:3.5}. 
Without loss of generality, let $z$ be a $2$-vertex. In a similar way, we can prove that $\eta\eta'\prec I$, so 
one of $y$ or $y'$ is a $2$-vertex, and moreover, $\eta\delta\gamma\nprec I$. If $y$ is a $2$-vertex then $\Lambda$ 
admits the following subcategory 
$$\xymatrix{\circ&z\ar[l]\ar[rd]^{\sigma}&&i&&& \\
&z'&b\ar[l]_{\sigma'}\ar[r]^{\delta^*}&x\ar[u]^{\gamma}& a\ar[l]_{\delta}&\ar[l]_{\eta} y\ar[r] & \circ} $$ 
which is a wild one-relation algebra RiIX. If $y'$ is a $2$-vertex we replace $\xymatrix{a&\ar[l]_{\eta} y\ar[r] & \circ}$ 
by $\xymatrix{a\ar[r]_{\eta'}& y' & \ar[l]\circ}$, which gives a wild subcategory of the same type. \medskip 

{\bf (4)} 
Suppose $a=c$, but $b\neq d$. In this case, we conclude from Lemma \ref{lem:3.6} that 
$\delta^*\gamma\alpha\nprec I$ and $\gamma\alpha\bar{\nu}\nprec I$. We can also assume that 
$i^*\neq \bar{i}$, since for $i^*=\bar{i}$, we would get the following wild subcategory of $\Lambda$: 
$$\xymatrix{&&i\ar[rd]^{\alpha}& \\ b\ar[r]_{\delta^*} &x\ar[ru]^{\gamma}\ar[r]_{\bar{\gamma}}&i^*\ar[r]_{\alpha^*}&j}$$
Note: $\bar{\gamma}\alpha^*\nprec I$, because $i$ is not a part of a triangle, and $\delta^*\bar{\gamma}\alpha^*\nprec I$, 
again by Lemma \ref{lem:3.6}, since we have no arrows $j\to b$. In particular, $i^*\neq\bar{i}$ implies also $\alpha^*\nu\delta\nprec I$ and $\nu\delta\bar{\gamma}\nprec I$. \medskip 

{\bf (5)} As the final  part of the proof,  we will show that $a$ is a $1$-vertex. Suppose to the contrary that $a$ is a $2$-vertex, we will
consider two cases.
Denote by $\eta,\eta'$ the arrows $\eta:y\to a$ and $\eta':a\to y'$. \smallskip 

We may assume further that $\nu\delta\nprec I$. Indeed, if $\nu\delta\prec I$, then by Lemma \ref{lem:3.3}, we get 
that there is an arrow $\bar{\gamma}=\alpha^*:x\to j$, so $i^*=x$ and $\bar{i}=j$, and $Q$ has locally the following shape 
$$\xymatrix@R=0.6cm{&&i\ar[rd]^{\alpha}&& \\ 
b\ar[r]^{\delta^*}&x\ar[rr]^{\bar{\gamma}}\ar[ru]^{\gamma}&&j & }$$
The above subquiver gives rise to a wild subcategory, because we have no arrows $j\to b$, so both 
$\delta^*\gamma\alpha\nprec I$ and $\delta^*\bar{\gamma}\nprec I$, due to \ref{lem:3.6} and \ref{lem:3.3}. \smallskip

{\bf Case A.} Suppose that $\alpha^*\bar{\nu}\nprec I$. \smallskip 

Observe first that $\alpha^*\nu\prec I$, since otherwise $\Lambda$ admits a wild subcategory of type $\wt{\wt{\bD}}_4$ 
formed by $\alpha,\alpha^*,\nu,\bar{\nu}$ and $\eta$. Moreover, if $i^*$ is a $1$-vertex, then due to Lemma \ref{lem:3.3}, 
$\alpha^*\nu\prec I$ implies an arrow $a\to i^*$, so $i^*$ is a $1$-vertex in a triangle. But then, by Proposition 
\ref{prop:4.2}, we have at least one $2$-regular predecessor of $j$, which leads to a contradiction, since here 
$j$ has two $1$-regular predecessors: $i$ and $i^*$. Hence we may assume that $i^*$ is $2$-regular. If also $d$ is 
$2$-regular, then $\Lambda$ admits the following wild subcategory of type $\widetilde{\widetilde{\bE}}_6$: 
$$\xymatrix@R=0.6cm{&&& \circ && \\ 
&&&\ar[u] i^*\ar[d]^{\alpha^*}&& \\ 
\bar{i}&x\ar[l]_{\bar{\gamma}}\ar[r]^{\gamma}&i\ar[r]^{\alpha}&j\ar[r]^{\bar{\nu}}&d&\circ\ar[l]}$$

From now on, let $d$ be a $1$-regular vertex and $\zeta:d\to d'$ be the unique arrow starting at $d$. Then $d'$ is a 
$2$-vertex, by Lemma \ref{lem:3.2}, so we have another arrow $d''\to d'$. Moreover, we have $\bar{\nu}\zeta\nprec I$, 
since otherwise $d'=i^*$, and we get a triangle $(\alpha^* \ \bar{\nu} \ \zeta)$ satisfying $\alpha^*\bar{\nu}\nprec I$ 
but $\bar{\nu}\zeta\prec I$. But this gives a contradition with Lemma \ref{lem:3.4}, because we have no double arrows 
from $i^*$ to $j$. Note also that $\alpha\bar{\nu}\zeta\nprec I$. In fact, if $\alpha\bar{\nu}\zeta\prec I$, we would 
get an arrow $d'\to i$, by Lemma \ref{lem:3.6}, hence $d'=x$, and so $d=a$ or $b$, a contradiction. Now, if also 
$\alpha^*\bar{\nu}\zeta\nprec I$, then $\Lambda$ admits the following wild subcategory of type RiVIII 
$$\xymatrix@R=0.4cm{\circ &i^*\ar[l]\ar[rd]^{\alpha^*}&i\ar[d]^{\alpha}&&& \\ 
\circ\ar[r]&a&j\ar[l]_{\nu}\ar[r]^{\bar{\nu}}&d\ar[r]^{\zeta}&d'& d''\ar[l]}$$ 

Finally, let $\alpha^*\bar{\nu}\zeta\prec I$. In this case, the vertex $i^*$ is the target of two arrows 
$\omega:d'\to i^*$ and $\omega^*:a\to i^*$. Hence, we get a wild subcategory of type $\wt{\wt{\bD}}$: 
$$\xymatrix{&i&&& d\ar[d]^{\zeta} & & \\ 
b\ar[r]^{\delta^*} &x \ar[u]^{\gamma}&a\ar[l]_{\delta}\ar[r]^{\omega^*}&i^*&d'\ar[l]_{\omega}\ar[r]^{\phi} 
& e & \ar[l]_{\psi} e'}$$ 
if $e$ is $2$-regular. For $e$ being $1$-regular with $\psi:e\to e'$ a unique arrow starting at $e$, we get similar 
wild subcategory of type $\wt{\wt{\bD}}$, because both $\phi\psi\nprec I$ and $\zeta\phi\psi\nprec I$. Indeed $\phi\psi\prec I$, 
then we get a triangle $d'\to e\to e'=d'' \to d'$ with a $1$-vertex $e$. Hence, by Proposition \ref{prop:4.2}, we 
conclude that there is an arrow $e'\to d$, so $e'=j$, since $d$ is the unique $1$-regular predecessor of $d'$. But 
then $e=i$, so $d'=x$, and we obtain that $d=a$ or $b$, a contradiction. Similarily, if $\zeta\phi\psi\prec I$, then 
Lemma \ref{lem:3.6} yields an arrow $e'\to d$, so again $e'=j$, and we are done in this case. \medskip 

Using dual arguments, one can as above exclude case $\delta^*\bar{\gamma}\nprec I$. \medskip  

{\bf Case B}. We are now left with the case when both $\alpha^*\bar{\nu}\prec I$ and $\delta^*\bar{\gamma}\prec I$. 
Then it follows from Lemma \ref{lem:3.3} that there are arrows $\mu:d\to i^*$ and $\mu':\bar{i}\to b$. \smallskip 

Let first both $\nu\eta'\prec I$ and $\eta\delta\prec I$. Then $y'=i^*$ and $y=\bar{i}$, due to Lemma 
\ref{lem:3.3}. Moreover, if $\eta\eta'\nprec I$, then $\Lambda$ admits the following subcategory 
$$\xymatrix{ &&i&&&z& \\ 
\circ &b\ar[l]\ar[r]^{\delta^*}&x\ar[u]_{\gamma}&a\ar[l]_{\delta}\ar[r]^{\eta'}& 
i^*\ar[ru]^{\xi}&d\ar[l]_{\mu}\ar[r]&\circ}$$ 
where $\xi$ is the second arrow starting at $i^*$ (different from $\alpha^*$). Note that $\eta'\xi\nprec I$, 
because otherwise, we would get a triangle $(\eta' \ \xi \ \eta)$ with $\eta'\xi\prec I$, but $\eta\eta'\nprec I$, 
which contradicts properties described in Lemma \ref{lem:3.4}. Further, both $b,d$ must be $2$-regular, because 
otherwise by Proposition \ref{prop:4.2} (these vertices lie in triangles) we would obtain $j=\bar{i}$ or $i^*=x$, 
and then $b=a$ or $d$ or $d=a$ or $b$, which is impossible. If $\mu\xi\nprec I$, then this is already a wild 
subcategory of type $\wt{\wt{\bD}}$. For $\mu\xi\prec I$, the vertex $z$ is lying in a triangle, so it must be 
a $2$-vertex by Proposition \ref{prop:4.2}, because $i^*$ does not have a $1$-regular predecessor ($a,d$ are 
$2$-regular). Then the above subcategory can be extended using arrow $\circ\to z$ yielding a wild one-relation 
algebra RiXVIII. \smallskip 

Now, assume that $\eta\eta'\prec I$. Then, by \ref{lem:3.3}, we have an arrow $\xi:i^*\to\bar{i}$. As in case 
$\eta\eta'\nprec I$, we conclude from Proposition \ref{prop:4.2} that $b$ and $d$ are $2$-regular vertices. 
Moreover, also $\bar{i},i^*$ are $2$-vertices, because we have the two arrows $\bar{\gamma},\xi$ ending at 
$\bar{i}$ and the two arrows $\alpha^*,\xi$ starting at $i^*$ ($Q$ is regular). Let $\sigma:d'\to d$ be the 
second arrow $\sigma\neq \bar{\nu}$ ending at $d$ and $\beta:b'\to b$ the second arrow $\beta\neq\mu'$ ending 
at $b$. If $d'$ is a $2$-vertex, then we can find the following wild hereditary subcategory in covering: 
$$\xymatrix{&&&& i \ar[d]_{\alpha} &&& \\ 
b'\ar[r]^{\beta} & b & \ar[l]_{\mu'} \bar{i}=y \ar[r]^{\eta} & a & \ar[l]_{\nu} j \ar[r]^{\bar{\nu}} & d & 
\ar[l]_{\sigma} d' \ar[r] & \circ }$$ \smallskip 

Hence we may assume that $d'$ is $1$-regular. Denote by $\sigma':d''\to d'$ the unique arrow ending at $d'$. 
Then $\sigma'\sigma\prec I$, since otherwise, we could construct a wild hereditary subcategory of the same 
type as above (replacing the arrow $d'\to \circ $ by the arrow $\sigma':d''\to d'$). Consequently, $d'$ must 
be a $1$-vertex in a triangle $d''\to d' \to d \to d''$, hence it follows from Proposition \ref{prop:4.2} that 
$d$ admits a $1$-regular successor different from $d''$. But the second successor of $d$ is $t(\mu)=i^*$, which 
is $2$-regular, and we obtain a contradiction. \medskip 

Finally, it remains to consider the case when $\nu\eta'\nprec I$ or $\eta\delta\nprec I$. Since the arguments are 
dual, we will proceed only in case $\nu\eta'\nprec I$. As before, we can see that $b$ and $d$ are $2$-regular 
vertices. But then, using that also $\nu\delta\nprec I$, we obtain the following 
wild subcategory:  
$$\xymatrix{ &&& d&& \\ &&& j\ar[d]_{\nu}\ar[u]^{\bar{\nu}}&& \\ 
\circ & \ar[l] b\ar[r]^{\delta^*}& x & a\ar[l]_{\delta} \ar[r]^{\eta'} & y' & e \ar@{-}[l]^{\phi} }$$ 
Indeed, if $y'$ is a $2$-vertex, then we take an arrow $\phi:e\to y'\neq \eta'$. Now, it is sufficient to see 
that for $y'$ being a $1$-vertex, we can take $\phi$ as the unique arrow $:y'\to e$. To see this, we have to 
prove that $\eta'\phi\nprec I$ and $\nu\eta'\phi\nprec I$. \smallskip 

By Lemma \ref{lem:3.3}, $\eta'\phi\prec I$ gives a triangle $a\to y' \to e \to a$, where $y'$ is a $1$-vertex. Then 
$e=y$ must be a $2$-vertex, due to Lemma \ref{lem:3.5}. But now, applying Proposition \ref{prop:4.2} to this triangle, 
we conclude that $a$ must have at least one $1$-regular predecessor, which is impossible, since $a$ has only 
$2$-regular predecessors, namely: $e=y$ and $j$. 

Finally, if $\nu\eta'\phi\prec I$, then by Lemma \ref{lem:3.6}, we have an arrow $e\to j$, so $e=i^*$, because 
$e$ is cannot be a $1$-vertex. This means we have an arrow $y'\to i^*$, hence $i^*$ is a $2$-vertex in $Q$. 
Take the second arrow $\psi\neq \alpha^*$, $\psi:i^*\to z$, starting at $i^*$. Further, $\alpha^*\nu\nprec I$, 
because otherwise, it follows from Lemma \ref{lem:3.3} that there is an arrow $a=c\to i^*$, and we obtain 
$a=d$ or $a=y'$, which is a contradiction. When $z$ is a $1$-vertex with the unique arrow $\psi':z\to z'$ 
starting at $z$, then $z'$ is $2$-regular, due to Lemma \ref{lem:3.2}. Moreover, then $\psi\psi'\nprec I$, since 
for $\psi\psi'\prec I$, we would obtain a triangle with vertices $i^*,z,z'=d$ (by \ref{lem:3.3}) with $1$-regular 
vertex $z$, and then Proposition \ref{prop:4.2} implies that $i^*$ has a $1$-regular predecessor $y'$ being 
a target of arrow starting at $d$, and hence $a=d$, because $\eta':a\to y'$ is the unique arrow ending at 
$1$-regular vertex $y'$. Consequently, in case $z$ is $1$-regular, we get the following subcategory 
$$\xymatrix{&  & i \ar[d]_{\alpha} & i^* \ar[ld]_{\alpha^*} \ar[r]^{\psi} & z \ar[r]^{\psi'} & z' & z'' \ar[l]_{\xi} \\ 
y \ar[r]^{\eta} & a & \ar[l]_{\nu} j \ar[r]^{\bar{\nu}} & d & \ar[l]_{} d' && }$$ 
which is isomorphic to a wild one-relation algebra RiXIII. For $z$ being a $2$-vertex, we obtain a wild 
subcategory of the same type, but with arrows 
$$\xymatrix{z & \ar[l]_{\psi^*} z' \ar[r]^{\xi} & z''}$$ 
instead of $\xymatrix{z \ar[r]^{\psi'} & z' & z'' \ar[l]_{\xi}}$. Indeed, $z$ is $2$-regular, so we can take the 
second arrow $\psi^*:z'\to z$ ending at $z$, which is different from 
$\psi$. Then it is sufficient to see that $z'$ is $2$-regular, and we will conclude existence of the second arrow 
$\xi:z'\to z''$, $\xi\neq \psi^*$. So suppose to the contrary that $z'$ is $1$-regular with the unique arrow ending 
at $z'$ denoted by $\xi:z''\to z'$. In this case, we would obtain the same wild subcategory as above if $\xi\psi^*\nprec I$, 
so we can assume that $\xi\psi^*\prec I$. Therefore, by Lemma \ref{lem:3.3}, we conclude that there is an arrow 
$z\to z''$ and vertices $z'',z',z$ lie on a triangle with $1$-vertex $z'$. As a result, using Proposition \ref{prop:4.2} 
again, we infer that there is also an arrow $z''\to i^*$, and then $z''=d$ or $z''=y'$. Because Lemma \ref{lem:3.2} 
implies that $z''$ is $2$-regular, we can have only $z''=d$. Once again, by Proposition \ref{prop:4.2}, we conclude 
that $d$ must have at least one $1$-regular predecessor, which is impossible, since $d$ has only $2$-regular 
predecessors $j$ and $z\neq j$.

\end{document}